\documentclass[11pt]{amsart}
\usepackage{amsmath, amssymb, verbatim,color, amscd}

\newtheorem{theorem}{Theorem}[section]

\newtheorem{corollary}[theorem]{Corollary}
\newtheorem{proposition}[theorem]{Proposition}

\renewcommand{\leq}{\leqslant}
\renewcommand{\geq}{\geqslant}

\theoremstyle{definition}

\theoremstyle{definition}

\numberwithin{equation}{section}

\newcommand{\ov}[1]{\overline{#1}}
\newcommand{\de}{\partial}
\newcommand{\db}{\overline{\partial}}

\renewcommand{\leq}{\leqslant}
\renewcommand{\geq}{\geqslant}

\numberwithin{equation}{section} \numberwithin{figure}{section}

\title[Degeneration  of Ricci-flat Calabi-Yau manifolds]{Degeneration  of Ricci-flat Calabi-Yau manifolds and its applications}
 \author[Y. Zhang]{Yuguang Zhang}
 \thanks{The  author is supported in part by  grant NSFC-11271015.}
\address{Yau Mathematical Sciences Center,  Tsinghua University,  Beijing 100084, P.R.China.}
\email{yuguangzhang76@yahoo.com}

\begin{document}
\begin{abstract}
This is a survey article of the recent progresses on the metric behaviour of Ricci-flat K\"{a}hler-Einstein  metrics along    degenerations of Calabi-Yau manifolds.
\end{abstract}
\maketitle
\section{Introduction}
A Calabi-Yau manifold $X$   is a   simply connected complex  projective manifold with trivial canonical bundle $\varpi_{X}\cong \mathcal{O}_{X}$, and  a polarized Calabi-Yau manifold $(X,L)$ is a Calabi-Yau manifold $X$ with an ample line bundle $L$. Note that the definition of Calabi-Yau manifold varies according to the documents, and we use this one for our convenience.
 For example, if $\mathfrak{f}$ is a generic homogenous polynomial of degree ${\rm deg} \mathfrak{f}= n+1$, $n\geq 3$,  such that the hypersurface  $$X=\{[z_{0},\cdots,z_{n}]\in \mathbb{CP}^{n}| \mathfrak{f}(z_{0},\cdots,z_{n})=0\}$$ is smooth, then $X$ is a Calabi-Yau manifold, and the restriction  $\mathcal{O}_{ \mathbb{CP}^{n}}(1)|_{X}$ of the  hyperplane   bundle  is an ample line bundle. The study of Calabi-Yau manifolds became
     very important  in the last four  decades, and   readers are referred to  the
survey articles   and   books \cite{Yau-survey,Yau2, GHJ,Hubsch}  for the  complete aspects   on  Calabi-Yau manifolds.  The present paper surveys the recent progresses on  metric  properties of Calabi-Yau manifolds.

In the 1970's, S.-T. Yau  proved the  Calabi's conjecture
in  \cite{Ya1}, which asserts the existence of Ricci-flat K\"{a}hler-Einstein  metrics on Calabi-Yau manifolds.

  \begin{theorem}[Calabi-Yau Theorem]\label{calabi-yau}
 If $X$ is a Calabi-Yau manifold, then
 for  any K\"{a}hler class $\alpha\in H^{1,1}(X, \mathbb{R})$,
there exists a
     unique Ricci-flat K\"{a}hler-Einstein  metric
     $\omega\in\alpha$, i.e. $${\rm Ric}(\omega)\equiv 0.$$
     \end{theorem}

      This  theorem is obtained by showing the existence and the  uniqueness of  solution of   Monge-Amp\`ere equation $$ (\omega_{0}+\sqrt{-1}\partial\overline{\partial }\varphi)^{n}=(-1)^{\frac{n^{2}}{2}}\Omega\wedge\overline{\Omega}, \  \  \ \sup_{X} \varphi=0,  $$  for a given   background K\"{a}hler metric $\omega_{0}\in \alpha$, where   $\Omega$ is a  holomorphic volume form, i.e. a nowhere vanishing section of    $\varpi_{X}$.
      The K\"{a}hler metric   $\omega=\omega_{0}+\sqrt{-1}\partial\overline{\partial }\varphi$ has vanishing Ricci curvature,  and
   the  Riemannian holonomy group of $\omega$ is a subgroup of  $SU(n)$.  Conversely,  a simply connected  Riemannian manifold with holonomy group a subgroup of  $SU(n)$ is a K\"{a}hler manifold with trivial canonical bundle, and   if  it is projective,   it is a Calabi-Yau manifold by our definition.   In fact, Calabi-Yau manifolds play some roles in  physics, precisely the string theory,  because the holonomy group of  Ricci-flat Calabi-Yau threefolds  is $SU(3)$ (cf. \cite{Hubsch}). It makes perfect sense to   understand the interaction between the  metric properties of Ricci-flat K\"{a}hler-Einstein  metrics and the  algebro-geometric properties of  underlying Calabi-Yau manifolds.

  Denote $\mathfrak{M}(\tilde{X})$  the set of all Ricci-flat K\"{a}hler-Einstein  metrics on Calabi-Yau $n$-manifolds with the same underlying  differential manifold $\tilde{X}$.   There are two parameters of $\mathfrak{M}(\tilde{X})$, one is the complex structure and the other is the K\"{a}hler class.
On a Calabi-Yau manifold $X$,  the Calabi-Yau theorem shows that  Ricci-flat K\"{a}hler-Einstein  metrics  are parameterized by the K\"{a}hler cone $\mathbb{K}_{X}\subset H^{1,1}(X, \mathbb{R})$, which is an open cone in $H^{1,1}(X, \mathbb{R})$ such that $\alpha \in \mathbb{K}_{X} $ if and only if
   there  is  a  K\"{a}hler  metric  representing $ \alpha $. A   deformation of $X$ means  a smooth proper   morphism $\pi: \mathcal{X}\rightarrow S$ for two  complex analytic varieties   $\mathcal{X}$ and $S$ such that there is a distinguished point $0\in S $  satisfying $X=\pi^{-1}(0)$. A universal deformation $\pi: \mathcal{X}\rightarrow S$ is a deformation of $X$ such that  any deformation $\pi': \mathcal{X}'\rightarrow S'$ of $X$ is isomorphic to  the pull-back under a unique morphism $\phi: S'\rightarrow S$.
     The Bogomolov-Tian-Todorov unobtructedness theorem (cf. \cite{Ti,Tod})  asserts that the local universal  deformation space of $X$ exists, and is smooth with tangent space $H^{1}(X, \mathcal{T}_{X})\cong H^{n-1,1}(X)$.
      Thus loosely speaking, $\mathfrak{M}(\tilde{X})$ is a space of real  dimension $h^{1,1}+2h^{n-1,1}$, and the tangent space at $X$ is $H^{1,1}(X, \mathbb{R})\oplus H^{n-1,1}(X) $.

    If  $\pi: \mathcal{X}\rightarrow S$ is   a deformation of a Calabi-Yau manifold $X$, and $\gamma: [0,1] \rightarrow \bigcup\limits_{t\in S} \mathbb{K}_{X_{t}}$ is a curve  where $X_{t}=\pi^{-1}(t)$, then we have a family of    Ricci-flat K\"{a}hler-Einstein Calabi-Yau manifolds $(X_{\gamma(s)}, \omega_{\gamma(s)})$, $s\in [0,1]$,
     with the same underlying differential manifold $\tilde{X}$ by the Calabi-Yau theroem.    Clearly, $\omega_{\gamma(s)}$ is a family of Ricci-flat K\"{a}hler-Einstein metrics on $\tilde{X}$ depending on the parameter $s\in [0,1]$ continuously.   However, if we degenerate $X$ to some singular variety by deforming the complex structure  on $X$, or degenerate the   K\"{a}hler classes to a non-K\"{a}hler class, the question arisen is how   those   Ricci-flat K\"{a}hler-Einstein metrics behave  along such degenerations.
     We certainly do not expect that those metrics converge  to another Ricci-flat K\"{a}hler-Einstein metric on $\tilde{X}$,  since the degeneration limit would not support any K\"{a}hler metric in the usual sense anymore, and
     some dramatic phenomena  must occur.

     There are several  motivations to  study this question.  Firstly, the Calabi-Yau theorem asserts the existence of  Ricci-flat K\"{a}hler-Einstein metrics by solving the  Monge-Amp\`ere equation,  a fully nonlinear partial differential equation, and
     however, very few of them can be written down explicitly.  It is desirable to improve our knowledge of Ricci-flat K\"{a}hler-Einstein  metrics, for example what the manifold with these metrics looks
   like. One method to achieve this understanding is by studying the limiting  behaviour of metrics along degenerations (cf. Section 2.1).
     Secondly, the  conifold transition (or more general extremal  transition) provides  a way to connect Calabi-Yau
     threefolds with different topology in algebraic geometry  (cf. \cite{Ro}).    It  was conjectured by physicists   Candelas and  de la Ossa  that this process is continuous in the space
    of all  Ricci-flat K\"{a}hler-Einstein threefolds  in  \cite{Ca}.   Therefore it is important to study how Ricci-flat K\"{a}hler-Einstein  metrics change in this process.
     Thirdly, there is a nice corresponding between the  compactifications of the moduli spaces of curves in the algebro-geometric sense and the differential geometric counterpart. It is  interesting to generalize such compactification to the case of
          Calabi-Yau manifolds.  Finally,  Strominger, Yau and Zaslow proposed  a  conjecture  in  \cite{SYZ}, so called SYZ conjecture,  for constructing mirror Calabi-Yau manifolds via dual special lagrangian fibration.  Later, a new version of   SYZ conjecture was proposed by   Kontsevich,  Soibelman, Gross and Wilson
    (cf. \cite{GW,KS,KS2}) by using   the collapsing of Ricci-flat K\"{a}hler-Einstein metrics. In both cases, it is needed to understand the metric behaviour of Ricci-flat K\"{a}hler-Einstein metrics when  complex structures degenerate to  the   large complex limit.

     Now we have a relatively clear picture  of the so called non-collapsing limit behaviour of Ricci-flat K\"{a}hler-Einstein metrics due to many recent works by various authors (cf. \cite{RuZ,RoZ1,RoZ2,To1,Song1,Song2,DS} etc.).  We say    a sequence $(X_{k},g_{k})$ of Riemannian manifolds non-collapsed, if there is a constant $\kappa>0$ such that
        $${\rm Vol}_{g_{k}}(B_{g_{k}}(p,1))\geq \kappa,$$ for any   metric 1-ball $B_{g_{k}}(p,1)\subset X_{k}$, and  otherwise,  we say $(X_{k},g_{k})$  collapsed.
        Roughly speaking, our current understanding can be summed up as the following:   the non-collapsing of Ricci-flat K\"{a}hler-Einstein metrics in the differential geometric sense is equivalent to the degeneration of Calabi-Yau manifolds to certain irreducible normal varieties  in some  algebro-geometric sense.
         The goal of this  paper is to  survey these   progresses.

   This paper is organized  as the following.  Section 2 is the preparation for  the  results reviewed in this paper. In Subsection 2.1, we recall an  earlier theorem due to R. Kobayashi on K3 surfaces.  Many later works  in this paper, more or less,  can be regarded as generalizations of this  result. In Subsection 2.2, we recall the generalized Calabi-Yau theorem for normal varieties due to Eyssidieux,   Guedj, and  Zeriahi, and in Subsection 2.3, we review Cheeger-Colding's theory on the Gromov-Hausdorff convergence of Riemannian manifolds with bounded Ricci curvature.  Section 3 studies  the convergence of Ricci-flat K\"{a}hler-Einstein metrics along degenerations of Calabi-Yau manifolds.  In Subsection 3.1, we survey the papers \cite{RoZ1,RoZ2}, which study the   convergence of  metrics under the assumption of degenerations, and in Subsection 3.2, we review a theorem due to Donaldson and Sun \cite{DS}, which study the algebro-geometric property of  limits  in the Gromov-Hausdorff convergence.  Subsection 3.3 establishes the equivalence  between the Gromov-Hausdorff convergence and the
 algebro-geometric degeneration in the non-collapsing case.  Those results are from \cite{To,Ta,Wang1}.  In Section 4, we apply the works reviewed in Section 3 to obtain a completion for the moduli space of polarized Calabi-Yau manifolds, which is  in \cite{zha1}.  Section 5 studies  the effect of  surgeries in algebraic geometry on Ricci-flat K\"{a}hler-Einstein metrics.  In Subsection 5.1, we recall the theorem in \cite{To1} first, which shows the convergence of metrics along degenerations of K\"{a}hler classes  while keeping the complex structure fixed.  Then we recall the result of  \cite{RoZ1}  for  the metric behaviour under flops, and apply it to the connectedness property of birational equivalent  Calabi-Yau threefolds.  Subsection 5.2  surveys   the  metric behaviour of Ricci-flat K\"{a}hler-Einstein metrics under extremal transitions in \cite{RoZ1}, which proves the conjecture due to Candelas and  de la Ossa.  In Section 6, we briefly  review the collapsing of Calabi-Yau manifolds, where our knowledge is much less complete  than the non-collapsing case. Finally,  we give an analytic proof of  a   finiteness theorem  for  polarized  complex manifolds in Appendix A, which is a joint work with Valentino Tosatti.

\noindent {\bf Acknowledgements:} The author would  like to thank Xiaochun Rong, Wei-Dong Ruan, Mark Gross and Valentino Tosatti for collaborations,  which lead to
many  of works surveyed in this paper.    This article is an extension of the  invited talk that the author gave at the conference 'Uniformization, Riemann-Hilbert Correspondence, Calabi-Yau Manifolds, and Picard-Fuchs Equations' at  Institut Mittag-Leffler.  The author  thanks  the organizers for the invitation, and  Institut Mittag-Leffler for the hospitality. Finally, the author  thanks Gross, Tosatti,  Takayama and Zhenlei Zhang  for some comments.

 \section{Preliminaries }
  \subsection{Ricci-flat K\"{a}hler-Einstein metrics on  K3 surfaces}
  In this subsection, we recall some earlier works  on Ricci-flat K\"{a}hler-Einstein metrics on  K3 surfaces,  which stimulate  the recent works surveyed in this paper (at least those works of the author).

 Let $T$ be a complex two torus, i.e. $T=\mathbb{C}^{2}/\Lambda$ where $\Lambda\cong \mathbb{Z}^{4}$ is a lattice.  The Kummer K3 surface $X$ is constructed as the following.   There is  a $\mathbb{Z}_{2}$-action on $T$ by $-1: (z_{1},z_{2})\mapsto (-z_{1},-z_{2})$, and the quotient $X=T/\mathbb{Z}_{2}$ is a complex orbifold with 16 ordinary double points as singularities.    The holomorphic $2$-form $\Omega=dz_{1}\wedge dz_{2}$ is invariant under this  $\mathbb{Z}_{2}$-action, and hence descends a holomorphic volume form on $X$ in the orbifold sense.
  We say $X$ a Calabi-Yau orbifold.
 The crepant resolution $\bar{\pi}:\bar{X} \rightarrow X$ is obtained by replacing the singular points by  ($-2$)-curves $E_{1}, \cdots, E_{16}$, which is a K3 surface called the Kummer K3 surface.  Note that there is flat orbifold K\"{a}hler metric $\omega$ on $X$ induced by the standard flat metric on the torus  $T$, and $\bar{\pi}^{*}\omega$ represents a non-K\"{a}hler  semi-ample class.  And  $$\alpha_{s}=[\bar{\pi}^{*}\omega]-s\sum_{i=1}^{16}a_{i} E_{i}, \  \  \  s\in (0, 1], $$ is a K\"{a}hler   class for $0<a_{i} \ll 1$, $i=1, \cdots, 16$.

 The Calabi-Yau theorem shows  the existence of K\"{a}hler-Einstein metrics for K3 surfaces.   In the case of Kummer K3 surface, there is an alternative  proof without using partial differential equations  due to Topiwala in \cite{Top1} (See also \cite{LubS} for  the related work), where the twistor theory and the deformation theory of complex structure are used.    It was expected that the curvature of the Ricci-flat K\"{a}hler-Einstein metric representing $\alpha_{s}$ on a Kummer K3 surface would concentrate around the 16 ($-2$)-curves (cf. \cite{GHawking,Pag}) when $s\ll 1$, and it was verified by  R.  Kobayashi in
 \cite{Kob1}, in which
  the limit behaviour of Ricci-flat K\"{a}hler-Einstein metrics  $\bar{\omega}_{s}\in \alpha_{s}$ is studied  when $s \rightarrow 0$ by a    gluing argument.

  Let's recall this construction.
    Locally each singular point of $X$ is of the form $ \mathbb{C}^{2}/ \mathbb{Z}_{2}$, the quadric cone in $\mathbb{C}^{3}$, and the crepant resolution is the total space of $\mathcal{O}_{\mathbb{CP}^{1}}(-2)$, which is diffeomorphic to the total space of $T^{*}S^{2}$.  We denote $E$ the $(-2)$-curve in $\mathcal{O}_{\mathbb{CP}^{1}}(-2)$, and we identify $\mathcal{O}_{\mathbb{CP}^{1}}(-2)\backslash E$ with $ (\mathbb{C}^{2}/ \mathbb{Z}_{2})\backslash \{0\}$ by the resolution map. Let $\rho=|z_{1}|^{2}+|z_{2}|^{2}$ on $\mathbb{C}^{2}$, which descents to a function on  $ \mathbb{C}^{2}/ \mathbb{Z}_{2}$ denoted still by $\rho$.  For any $0<a \leq 1$, if we let  $$f_{a}(\rho)=\rho \sqrt{1+\frac{a^{2}}{\rho^{2}}}+a \log (\frac{\rho}{a+\sqrt{a^{2}+\rho^{2}}}), $$ then $\omega_{a}=\sqrt{-1}\partial\overline{\partial}f_{a}$ defines a complete asymptotically locally Euclidean (ALE)  Ricci-flat K\"{a}hler-Einstein metric on $\mathcal{O}_{\mathbb{CP}^{1}}(-2)$,  called the Eguchi-Hanson metric.   Note that  $[\omega_{a}]=-2a E$, and $\omega_{a}$ converges smoothly to the flat metric $\omega_{flat}=\sqrt{-1}\partial\overline{\partial} \rho$ of $ \mathbb{C}^{2}/ \mathbb{Z}_{2}$ on any compact subset of  $\mathcal{O}_{\mathbb{CP}^{1}}(-2)\backslash E$ when $a \rightarrow 0$, and $\omega_{a} \sim \omega_{flat}$ for $\rho\gg 1$.  The curvature of the restriction of $\omega_{a}$ on $E$ is $a^{-1}$, and the volume is $4\pi a$.

 In \cite{Kob1}, a background K\"{a}hler metric $\omega_{s}$ is constructed such that $\omega_{s}=\omega_{\frac{sa_{i}}{2}}$  on a small neighborhood of $E_{i}$, $i=1, \cdots, 16$, and $\omega_{s}=\omega$   on a big compact subset of $X_{reg}=\bar{X}\backslash \bigcup\limits_{i=1}^{16}E_{i}$. The metric $\omega_{s}$ is  an approximate Ricci-flat K\"{a}hler metric.
  Let $\varphi_{s}$ be the unique potential function such that $\sup\limits_{\bar{X}} \varphi_{s}=0$, and $\bar{\omega}_{s}=\omega_{s}+\sqrt{-1}\partial\overline{\partial}\varphi_{s}$ be  the Ricci-flat  K\"{a}hler-Einstein metric.  By using Yau's estimate of  the  Monge-Amp\`ere equation \cite{Y3}, the  explicit bounds are  obtained in  \cite{Kob1}, i.e.  \begin{itemize}
  \item[i)] $C^{0}$-bound:   $\|\varphi_{s}\|_{C^{0}}\leq C |s|^{2}$,  \item[ii)] $C^{2}$-bound:  $$ (1-C|s|^{\frac{1}{2}})\omega_{s} \leq \bar{\omega}_{s} \leq (1+C|s|^{\frac{1}{2}})\omega_{s},$$ for
 a constant $C>0$ independent of $s$, \item[iii)] higher order bounds:
  $\|\varphi_{s}\|_{C^{k}(K)}\leq C_{k,K} |s|^{2}$, for  constants $C_{k,K}>0$ on any compact subset $K\subset X_{reg}$. \end{itemize}
      As a consequence, $\bar{\omega}_{s}$ converges smoothly  to the standard flat metric $\omega$ on any  compact subset $K\subset X_{reg}$ when $s\rightarrow 0$, and the exceptional divisors $E_{i}$ shrink to points.  It is clear that $(\bar{X}, \bar{\omega}_{s})$ converges to $(X, \omega)$ in the Gromov-Hausdorff sense (cf. Subsection 2.3).  Furthermore, the bubbling limits are also studied in  \cite{Kob1}, and it is proved  that near any $E_{i}$,  $s^{-1}\bar{\omega}_{s}$ converges the the Eguchi-Hanson metric $\omega_{\frac{a_{i}}{2}}$ on $\mathcal{O}_{\mathbb{CP}^{1}}(-2)$ in a certain sense when $s\rightarrow 0$.

  The above convergence result is also proved for more general K3 orbifold $X$
  in \cite{Kob1} with a little weaker bounds for the metric by gluing ALE gravitational instantons obtained by Kronheimer (cf. \cite{Kron}).  Moreover, since any Ricci-flat K\"{a}hler-Einstein metric on a K3 surface is HyperK\"{a}hler,  this result can also be explained as
   the  convergence of metrics  along a degeneration of complex structures      by performing   the HyperK\"{a}hler rotation.     Denote $I$ the complex structure of $\bar{X}$, and let $\Omega_{s}$ be the holomorphic volume form on $\bar{X}$ with  $\bar{\omega}_{s}^{2}=\Omega_{s}\wedge\overline{\Omega}_{s}$.  For any $s\in (0,1]$, we have a  new complex structure $J_{s}$  with corresponding  K\"{a}hler form and holomorphic volume form $$ \bar{\omega}_{J_{s}}= {\rm Re}\Omega_{s},  \  \  \ \Omega_{J_{s}}= {\rm Im}\Omega_{s}+\sqrt{-1} \bar{\omega}_{s},$$  and however  the respective Riemannian metric does not change. Thus we obtain the convergence of metrics along the degeneration of complex structures  $J_{s}$.  As an application, it is shown in \cite{Kob1} and \cite{KT} that  one can fill the holes in the moduli space of K\"{a}hler K3 surfaces by some K\"{a}hler K3 orbifolds.

  The analogue of the above construction was also carried out for a rigid Calabi-Yau threefold in \cite{Lu1}.   \cite{Do} gives a different proof of the result for the Kummer K3 surface with less PDE involved.
    More importantly, Gross and Wilson used the gluing argument to prove a theorem about the  collapsing of Ricci-flat K\"{a}hler-Einstein metrics on  K3 surfaces in \cite{GW}. Here the K3 surface $X$ is assumed to admit  an elliptic fibration $f:X\rightarrow B$ with 24  singular fibers of  type $I_{1}$. The background K\"{a}hler  metric is constructed by gluing the semi-flat K\"{a}hler-Einstein metric on the part with smooth fibers, and the Ooguri-Vafa metric near the singular fibers. Then Yau's estimate of  the  Monge-Amp\`ere equation gives the proof  in \cite{GW}. See Section 6 for more discussions.  Furthermore, the similar gluing arguments were used by D. Joyce to construct compact $G_{2}$ and $Spin(7)$ manifolds (cf. \cite{Joy}).

  If we want to generalize  the above convergence  theorem to   higher dimensional Calabi-Yau manifolds,  which have  more complicated singularities, the gluing arguments fail due to many difficulties. One of them is the lack of local model like the  Eguchi-Hanson metric in the above case.  However, recent progresses on the Monge-Amp\`ere equation and the  Gromov-Hausdorff theory make it possible, if we are willing to sacrifice some explicit estimates.

 \subsection{Calabi-Yau variety}
 In this subsection, we recall the generalized Calabi-Yau theorem   due to  Eyssidieux,   Guedj, and Zeriahi, which asserts  the existence of   singular K\"{a}hler-Einstein metrics for certain  normal varieties.

Let's recall some notions  first.
  A normal projective variety $X$ is called  $1$-Gorenstein if  the dualizing   sheaf $\varpi_{X}$ is an invertible  sheaf, i.e. a line bundle,  and is called Gorenstein if furthermore $X$ is Cohen-Macaulay. A variety $X$ has only canonical singularities if $X$ is $1$-Gorenstein, and for any resolution $\bar{\pi}: \bar{X}\rightarrow X$, $\bar{\pi}_{*}\varpi_{\bar{X}}=\varpi_{X}$, which is equivalent to that the canonical divisor $\mathcal{K}_{X}$ is Cartier, and $$\mathcal{K}_{\bar{X}} = \bar{\pi}^{*}\mathcal{K}_{X}+\sum\limits_{E} a_{E}E, \ \ {\rm and} \  \  a_{E} \geq 0, $$ where $E$ are effective  exceptional prime divisors.
    The ordinary double point, i.e.  the singularity locally given by $$z_{0}^{2}+\cdots +z_{n}^{2}=0,  \  \   \ {\rm  in }  \  \  \mathbb{C}^{n+1} $$ for $n\geq 2$, is a simple example of canonical singularity.
   If $X$ has only canonical singularities, then the singularities are rational, and $X$ is Cohen-Macaulay (cf. (C)   of \cite[Section 3]{Re}), which implies that $X$ is Gorenstein.

    A Calabi-Yau variety $X$ is a normal projective Gorenstein variety with trivial dualizing   sheaf $\varpi_{X}\cong \mathcal{O}_{X}$, and having at most  canonical singularities. A  polarized Calabi-Yau variety $(X,L)$ is a Calabi-Yau variety $X$  with an ample line bundle $L$.  If $X$ has only finite ordinary double points as singularities, we call $X$ a conifold.

  There is a notion of K\"{a}hler metric for normal varieties (cf.  \cite[Section 5.2]{EGZ}).  Let  $X$ be a normal $n$-dimensional projective variety.
  For any singular point  $p\in  X_{sing}$ and a  small neighborhood $U_{p}\subset X$ of $p$,
a pluri-subharmonic function $v$ (resp. strictly pluri-subharmonic,
and pluri-harmonic) on $U_{p}$ is an upper semi-continuous function
with value in $\mathbb{R}\cup \{-\infty\}$ ($v$  is not locally
$-\infty$)  such that  $v$ extends to a pluri-subharmonic function $
\tilde{v}$ (resp. strictly pluri-subharmonic, and pluri-harmonic) on
a neighborhood of the image of  some local embedding $U_{p}
\hookrightarrow \mathbb{C}^{m}$.  We call $v$ smooth if  $
 \tilde{v}$ is smooth.
     A   form $\omega$  on $X$
       is called a K\"{a}hler metric, if $\omega$  is a smooth K\"{a}hler metric   in the usual
sense on $X_{reg}$ and, for any  $p\in X_{sing}$, there is a
neighborhood $U_{p}$  and a continuous strictly pluri-subharmonic
function $v$ on
 $U_{p}$ such that
 $\omega=\sqrt{-1}\partial\overline{\partial}v$ on $U_{p}\cap X_{reg}
 $. We call $\omega$  smooth if $ v$ is
 smooth in the  above sense. Otherwise, we call $\omega$ a singular K\"{a}hler metric.
 If $\mathcal{PH}_{X}$ denotes the sheaf of pluri-harmonic
functions on $X$, then  any K\"{a}hler metric  $\omega$ represents
a class $[\omega]$ in $H^{1}(X, \mathcal{PH}_{X})$ (cf.
Section 5.2 in \cite{EGZ}).  Note that $H^{1}(X, \mathcal{PH}_{X})\cong H^{1,1}(X,\mathbb{R})$ if $X$ is  smooth.    We also have an  analogue of Chern-Weil
theory for line bundles on $X$ (see \cite{EGZ} for details).

 In  \cite{EGZ}, a   generalized Calabi-Yau theorem is obtained for   Calabi-Yau varieties, i.e. the existence and the uniqueness of  singular Ricci-flat K\"{a}hler-Einstein metrics  with bounded potentials.  The   potential function   in the theorem  is proved to be  continuous  in \cite{EGZ2}.

    \begin{theorem}[\cite{EGZ}]
  \label{EGZ}
  Let $X$ be a  Calabi-Yau
  $n$-variety,
   and  $\omega_{0}$ be a smooth  K\"{a}hler form on $X$.  Then  there is a unique  continuous    function
 $\varphi$ on $X$   satisfying the following Monge-Amp\`ere equation  $$(\omega_{0}+ \sqrt{-1}\partial
 \overline{\partial} \varphi)^{n}=(-1)^{\frac{n^2}{2}}\Omega \wedge
 \overline{\Omega}, \ \   \sup_{X} \varphi =0 , \  \  {\rm and} \  \varphi\geq - C,
 $$  where $\Omega$ is a  holomorphic volume form, i.e. a nowhere vanishing section of the dualizing sheaf   $\varpi_{X}$, such that $\int_{X}\omega_{0}^{n}=\int_{X}(-1)^{\frac{n^2}{2}}\Omega \wedge \overline{\Omega} $.
  The restriction of the singular K\"{a}hler metric  $\omega=\omega_{0}+\sqrt{-1}\partial\overline{\partial }\varphi$ on the regular locus $X_{reg}$  is a smooth  Ricci-flat K\"{a}hler-Einstein metric, and  $\omega\in [\omega_{0}]\in H^{1}(X,
 \mathcal{PH}_{X})$.
   \end{theorem}

    The K\"{a}hler-Einstein metric $\omega$ gives a metric space structure on $X_{reg}$, and however, unlike the smooth Calabi-Yau case, it is unclear whether $\omega$ induces a compact metric space structure on $X$. If $X$ has only orbifold  singularities, then $\omega$ coincides with the orbifold   Ricci-flat K\"{a}hler-Einstein metric obtained previously in \cite{KT}, which induces a metric space structure on $X$.

    In the general case   of Theorem \ref{EGZ},  little is known for  the asymptotic behaviour of  $\omega$ and $\varphi$ near the singular locus $X_{sing}$.  Assume that $X$ is a conifold, i.e. $X$ has only ordinary double points as singularities.  For any $p\in X_{sing}$, there is a neighborhood $U_{p}\subset X$ isomorphic to an open subset of the quadric $\mathcal{Q}=\{(z_{0}, \cdots z_{n})\in \mathbb{C}^{n+1} |z_{0}^{2}+\cdots +z_{n}^{2}=0\}$.  A complete Ricci-flat K\"{a}hler-Einstein metric $\omega_{co}$ is constructed on     $\mathcal{Q}$ in \cite{Sten} and previously in \cite{Ca} for $n=3$, which is a cone metric.   More precisely,  if $\rho=|z_{0}|^{2}+\cdots +|z_{n}|^{2}$, then \begin{equation}\label{eco} \omega_{co}=\sqrt{-1}\partial\overline{\partial} \rho^{1-\frac{1}{n}}.\end{equation} It is speculated (cf. \cite{Wi}) that  $\omega$ is asymptotic to $\omega_{co}$ near the singular point $p$ in a certain sense, for example, $\|\omega-\omega_{co}\|_{C^{0}(U_{p}, \omega_{co})} \leq C \rho^{\varepsilon}$ for some constants $C>0$ and $\varepsilon >0$.   It is clearly true for the case of $n=2$, in which $\omega_{co}$ is the standard flat orbifold metric on $\mathbb{C}^{2}/\mathbb{Z}_{2}$, and however, it is still open for the higher dimensional case.

    Now we  apply Theorem \ref{EGZ} to obtain some  canonical  embeddings of Calabi-Yau varieties.
     Let $X$ is a Calabi-Yau $n$-variety.     If $L$ is an ample line bundle on $X$,  there is an  $m>0$ such that $L^{m}$ is very ample, and $H^{i}(X, L^{\mu})=\{0\}$ for any $i>0$ and $\mu \geq m$.
  A basis $\Sigma=\{s_{0}, \cdots, s_{N}\}$ of $H^{0}(X, L^{m})$ gives an embedding $\Phi_{\Sigma}:X \hookrightarrow \mathbb{CP}^{N}$ by $x \mapsto [s_{0}(x), \cdots, s_{N}(x)]$, which satisfies  $L^{m}=\Phi_{\Sigma}^{*}\mathcal{O}_{\mathbb{CP}^{N}}(1)$,  where $N=\dim_{\mathbb{C}}H^{0}(X, L^{m})-1$.
 The pullback  $\omega_{\Sigma}=\Phi_{\Sigma}^{*}\omega_{FS}$ of the Fubini-Study metric is a smooth K\"{a}hler metric in the above sense such that $[\omega_{\Sigma}]=m c_{1}(L)\in NS_{\mathbb{R}}(X)$. The Hermitian metric $h_{FS}$ of $\mathcal{O}_{\mathbb{CP}^{N}}(1)$,  whose curvature is the Fubini-Study metric,  restricts to an  Hermitian metric $h_{\Sigma}=\Phi_{\Sigma}^{*}h_{FS}$ on $L^{m}$, which satisfies that $\omega_{\Sigma}=-\frac{\sqrt{-1}}{2}\partial\overline{\partial}\log |\vartheta|_{h_{\Sigma}}^{2}$ on $X_{reg}$ for any local section $\vartheta$ of $L^{m}$.   We   regard $\Phi_{\Sigma}(X)$ as a   point in    $\mathcal{H}il_{N}^{P}$, denoted still by  $\Phi_{\Sigma}(X)$, where   $\mathcal{H}il_{N}^{P}$ is the Hilbert scheme  parametrizing  subschemes of $\mathbb{CP}^{N}$ with the  Hilbert polynomial $P=P(k)=\chi (X,L^{mk})$.

  Let $\omega=\omega_{\Sigma}+\sqrt{-1}\partial\overline{\partial }\varphi$ be the
 singular   Ricci-flat K\"{a}hler-Einstein metric obtained in Theorem \ref{EGZ}. Note that  $\omega$ is  unique in $m c_{1}(L)$, and particularly is  independent of the choice of $\Phi_{\Sigma}$.    By the boundedness of $\varphi$, we have that $h=\exp(-\varphi) h_{\Sigma}$ is a singular   Hermitian metric on $L^{m}$ whose curvature is $\omega$.
We define an  $L^{2}$-norm $\| \cdot \|_{L^{2}(h)}$ on $H^{0}(X,L^{m})$ by
 $$\| s\|_{L^{2}(h)}^{2}=\int_{X}|s|^{2}_{h} \omega^{n}=\int_{X}e^{-\varphi}|s|^{2}_{h_{\Sigma}} \omega^{n}. $$  If $h'$ is another Hermitian metric with the same curvature $\omega$, then $\partial\overline{\partial}\log \frac{h}{h'}\equiv 0$, i.e.   $\log \frac{h}{h'}$ is a pluriharmonic function on a closed normal variety $X$,  and thus $h=e^{\varsigma}h'$ for a constant $\varsigma$.  If $\Sigma_{h}=\{s_{0}, \cdots, s_{N}\}$ is an  orthonormal basis    of $H^{0}(X,L^{m})$ with respect  to $\| \cdot \|_{L^{2}(h)}$, then $\Sigma_{h'}=\{e^{-\frac{\varsigma}{2}}s_{0}, \cdots, e^{-\frac{\varsigma}{2}}s_{N}\}$ is  orthonormal with  respect  to $\| \cdot \|_{L^{2}(h')}$, and therefore  $\Sigma_{h}$ and $\Sigma_{h'}$ induce the same embedding $ \Phi_{\Sigma_{h}}= \Phi_{\Sigma_{h'}}$.

 If $\Sigma_{h}$ and $\Sigma_{h}'$ are
     two orthonormal bases   of $H^{0}(X,L)$ with  respect  to $\| \cdot \|_{L^{2}(h)}$, there is an  $u\in SU(N+1)\subset SL(N+1)$ such that  $[\Sigma_{h}]=[\Sigma_{h}']\cdot u$, $\Phi_{\Sigma_{h}'}(x)=\Phi_{\Sigma_{h}}(x)\cdot u$ for any $x\in X$, and thus $\Phi_{\Sigma_{h}'}(X)=\Phi_{\Sigma_{h}}(X)\cdot u$ in
     $\mathcal{H}il_{N}^{P}$, where $-\cdot-: \mathcal{H}il_{N}^{P} \times SL(N+1) \rightarrow \mathcal{H}il_{N}^{P} $ is the  $SL(N+1)$-action on $\mathcal{H}il_{N}^{P}$  induced by the natural $SL(N+1)$-action on $ \mathbb{CP}^{N}$.   We denote the
       $SU(N+1)$-orbit  \begin{equation}\label{eor}  RO(X,L^{m})=\{\Phi_{\Sigma_{h}}(X)\cdot u| u\in SU(N+1) \}\subset  (\mathcal{H}il_{N}^{P})_{red},   \end{equation} where $(\mathcal{H}il_{N}^{P})_{red}$ denotes the reduced variety of the Hilbert scheme $\mathcal{H}il_{N}^{P}$.    Note that  $ RO(X,L^{m})$  is compact under the analytic topology of $(\mathcal{H}il_{N}^{P})_{red}$,  and depends only on the  singular K\"{a}hler metric  $\omega$,  but not on the choice of $h$.

       The embedding constructed here will be used in Section 4 for the completion of the  moduli space of  polarized Calabi-Yau manifolds.

 \subsection{Gromov-Hausdorff topology}
 This subsection recalls  the notion and results of  Gromov-Hausdorff topology, which is the other main ingredient of the  works  surveyed in this paper.

  In  the 1980's, Gromov introduced the notion of  Gromov-Hausdorff
distance $d_{GH} $   between metric spaces in \cite{Gro}, which provides a frame to study families  of Riemannian manifolds.
 For any two compact metric spaces $A$ and $B$, the Gromov-Hausdorff distance of $A$ and $B$ is  $$d_{GH}(A,B)=\inf\{d_{H}^{Z}(A,B)| \exists \  {\rm isometric \ embeddings} \ A, B\hookrightarrow Z\}, $$ where $Z$ is a metric space,  $d_{H}^{Z}(A,B)$ is the standard  Hausdorff distance between $A$ and $B$ regarded as subsets by the isometric embeddings, and the infimum is taken for all possible $Z$ and isometric  embeddings. We denote $\mathcal{M}et$   the space of the isometric equivalent  classes of  all compact metric spaces equipped with the  topology, called the Gromov-Hausdorff topology,  induced by the Gromov-Hausdorff distance $d_{GH}$. Then $\mathcal{M}et$ is a complete metric space (cf. \cite{Gro,Rong}).

The Gromov-Hausdorff
distance is used to obtain   many important progresses of Riemannian geometry in the last three  decades, and the readers are referred to  the
survey articles \cite{Fu,Rong,C}  for these achievements.
 The following is the
  Gromov pre-compactness theorem:

   \begin{theorem}[\cite{G1}]\label{G1}
   Let $(X_{k}, g_{k})$ be a
   family of compact Riemannian
  manifolds such that Ricci curvatures and
  diameters satisfy   $${\rm Ric}(g_{k})\geq -C , \  \  \  {\rm diam}_{g_{k}}(X_{k}) \leq D, $$ for  constants  $C$ and $D$  independent of $k$. Then, a subsequence of $(X_{k},
  g_{k})$ converges to a compact  metric space $(Y, d_{Y})$  in the Gromov-Hausdorff
  sense.
  \end{theorem}

If we further assume that there is a lower bound of volumes in this theorem, i.e. ${\rm Vol}_{ \omega_{k}}( X_{k})\geq v>0$ for a  constant $v$, then the Bishop-Gromov comparison theorem shows
 $${\rm Vol}_{g_{k}}(B_{g_{k}}(p,1))\geq \kappa,$$ for a constant $\kappa>0$, and thus  $(X_{k}, g_{k})$ is non-collapsed.
   In the non-collapsing case,  the structure of the limit space $Y$  is studied
  in \cite{CC1,CC2,CT,CCT}, and many  properties of $Y$ is obtained, for example, the Hausdorff dimension of  $Y$ is  the same with that  of  $X_{k}$.

 Now let's focus on  compact  Ricci-flat  K\"{a}hler-Einstein  manifolds.
  The Gromov pre-compactness theorem  shows that a   family of
  compact Ricci-flat  K\"{a}hler-Einstein  manifolds with a  uniform upper bound of  diameters  converges to a compact  metric
  space by passing to a subsequence. Regularities of the limit is studied   in \cite{CCT,CT}, and it is shown that it  is almost a Ricci-flat manifold.

 \begin{theorem}[\cite{CCT,CT}]\label{CC}
 Let $(X_{k}, \omega_{k})$ be a
   family of compact  Ricci-flat  K\"{a}hler-Einstein  $n$-manifolds, and  $(Y,
   d_{Y})$ be a compact  metric space such that $$(X_{k}, \omega_{k})  \stackrel{d_{GH}}\longrightarrow
  (Y, d_{Y}).
   $$ If $${\rm Vol}_{ \omega_{k}}( X_{k})\geq v>0, \ \ {\rm and } \ \  c_{2}(X_{k})\cdot [ \omega_{k}]^{n-1}
   \leq C,$$ for  constants $v$ and  $C$
   independent of $k$, where  $ c_{2}(X_{k})$ is the second Chern-class of $X_{k}$, then   there is a closed subset $S\subset Y$ with
   Hausdorff dimension $\dim_{\mathcal{H}}S\leq 2n-4 $ such that $Y\backslash S
   $ is a Ricci-flat  K\"{a}hler-Einstein  $n$-manifold. Furthermore, off a subset of $S$ with $(2n-4)$-dimensional Hausdorff
measure zero, $S$ has only  orbifold type singularities
$\mathbb{C}^{n-2}\times \mathbb{C}^{2}/\Gamma$, where $\Gamma$ is a
finite subgroup of $ SU(2)$.
 \end{theorem}

 Actually, the convergence in this theorem is stronger than the Gromov-Hausdorff convergence. It is in the sense of Cheeger-Gromov on the regular locus, i.e. if we denote  $\omega_{\infty}$ the Ricci-flat  K\"{a}hler-Einstein on $Y\backslash S$ inducing the metric $d_{Y}$, and denote  $J_{k}$ (respectively $J_{\infty}$) the complex structure  of $X_{k}$ (respectively $Y\backslash S$), then for any compact subset $K\subset Y\backslash S$,   there are smooth embeddings $F_{K,k}: K\rightarrow X_{k}$ such that $$F_{K,k}^{*}\omega_{k} \rightarrow \omega_{\infty}, \  \  {\rm and} \  \  F_{K,k}^{*}J_{k}= dF_{K,k}^{-1}\circ J_{k} \circ dF_{K,k}\rightarrow J_{\infty},$$ in the $C^{\infty}$-sense.

 The tangent cone  $Y_{y}$ at  a point $y\in Y$ is defined as a pointed Gromov-Hausdorff limit of $(Y, \mu_{i}d_{Y},y)$ for a sequence $\mu_{i} \rightarrow+\infty$.  In the case of Theorem \ref{CC}, the  tangent cone is  proved to exist for any $y\in Y$ (cf. \cite{CC1,CC2}), and $S$ is defined as the locus where $Y_{y}$ is not isometric to $\mathbb{C}^{n}$.  More importantly, it is shown that any $Y_{y}$ is isometric to $C(M_{y})\times \mathbb{C}^{l}$ where $C(M_{y})$ is the metric cone over a
   compact metric space $M_{y}$ of Hausdorff dimension $2n-2l-1$  (cf. \cite{CC1,CC2,CCT}).

 When $n=2$, Theorem \ref{CC} was obtained previously in \cite{An0,An1,BKN,Na,Tia} etc., where $Y$ has only finite  isolated  orbifold  points as singularities, and $Y$ is clearly a complex  analytic variety. Moreover,  $Y$ is a K3 orbifold, and this fact was used  to study the moduli space of K3 surfaces in \cite{An0,An1}.  It is  expected  that $Y$ is also   a complex  analytic variety in the higher dimensional case, which was  proved by Donaldson and Sun under the assumption that those K\"{a}hler metrics $\omega_{k}$ belong to  integral K\"{a}hler classes (cf. \cite{DS}). We will continue the  discussion on  the Gromov-Hausdorff topology, and  review Donaldson-Sun's  theorem in Section 3.2.

 \section{Convergence of Ricci-flat K\"{a}hler-Einstein metrics}

In this section, we study the relationship between the degeneration in the algebro-geometric sense and the  non-collapsing convergence in the Gromov-Hausdorff sense.

\subsection{Degeneration}
This subsection study the limit behaviour of Ricci-flat  K\"{a}hler-Einstein   metrics along  algebro-geometric degenerations.

 A degeneration of  Calabi-Yau $n$-manifolds $\pi: \mathcal{X}\rightarrow \Delta$  is a flat morphism from a  variety $\mathcal{X}$ of dimension $n+1$  to a disc $ \Delta\subset \mathbb{C}$ such that for any $t\in \Delta^{*}=\Delta \backslash\{0\}$,  $X_{t}=\pi^{-1}(t)$ is a Calabi-Yau manifold,  and  the central fiber   $X_{0}=\pi^{-1}(0)$ is singular. If  there is a  relative  ample line bundle  $\mathcal{L}$   on $\mathcal{X}$, we call it a  degeneration  of polarized Calabi-Yau manifolds, denoted by   $(\pi: \mathcal{X}\rightarrow \Delta, \mathcal{L})$.
   If  we further  assume that  $X_{0}$ is  a Calabi-Yau  variety, then   the total space $\mathcal{X}$ is normal as  any fiber $X_{t}$ is reduced and normal. Thus
      the relative dualizing sheaf $\varpi_{\mathcal{X}/ \Delta}$  is defined, i.e. $\varpi_{\mathcal{X}/ \Delta}\cong \varpi_{\mathcal{X}}\otimes \pi^{*}\varpi_{\Delta}^{-1}$,  and is  trivial, i.e. $\varpi_{\mathcal{X}/ \Delta}\cong \mathcal{O}_{\mathcal{X}}$,   since every fiber is normal,  Cohen-Macaulay and Gorenstein.
    For example, if  $$\mathcal{X}=\{([z_{0},\cdots,z_{n}],t)\in \mathbb{CP}^{n}\times \Delta| t \mathfrak{f}(z_{0},\cdots,z_{n})+ \mathfrak{g}(z_{0},\cdots,z_{n})=0\}, $$ $ t\in \Delta \subset\mathbb{C}$, and  $\pi: \mathcal{X} \rightarrow\Delta$ is  the restriction of the  projection,  where $ \mathfrak{f}$ and $ \mathfrak{g}$  are generic homogenous polynomials  of degree $ n+1$ such that $X_{t}$ is smooth for any $t\in \Delta^{*}$, and $X_{0}$ is a Calabi-Yau variety, then $\pi: \mathcal{X}\rightarrow \Delta$ is a degeneration of Calabi-Yau manifolds, and the restriction $\mathcal{O}_{\mathbb{CP}^{n}}(1)|_{\mathcal{X}}$ of the hyperplane bundle is relative   ample.

Let  $(\pi: \mathcal{X}\rightarrow \Delta, \mathcal{L})$ be a  degeneration of polarized Calabi-Yau manifolds with the central fiber $X_{0}$
  a Calabi-Yau
variety,    and  $\omega_{t}\in c_{1}(\mathcal{L})|_{X_{t}}\in H^{1,1}(X_{t}, \mathbb{R})$, $t\in
\Delta^{*}$,  be  the unique Ricci-flat  K\"{a}hler-Einstein   metric on $X_{t}$.    The limit behaviour of $\omega_{t}$, when $t\rightarrow 0$, is studied in \cite{RuZ} and \cite{RoZ1}. In \cite{RuZ}, it is proved that $\omega_{t}$ converges to the singular Ricci-flat  K\"{a}hler-Einstein   metric $\omega$ obtained in Theorem  \ref{EGZ} under the technique assumption that $X_{0}$ is a Calabi-Yau conifold,  or the total space $\mathcal{X}$ is smooth.  These additional hypothesises  are  removed in
  \cite{RoZ1}, and more precisely, we obtain the following theorem.

\begin{theorem}[\cite{RoZ1}]\label{RZ1}
 Let $(\pi: \mathcal{X}\rightarrow \Delta, \mathcal{L})$ be a  degeneration of polarized Calabi-Yau manifolds  with the central fiber $X_{0}$
  a Calabi-Yau
$n$-variety.
If $\omega_{t}$ denotes the unique Ricci-flat  K\"{a}hler-Einstein   metric
in  $
c_{1}(\mathcal{L})|_{X_{t}}\in H^{1,1}(X_{t}, \mathbb{R})$, $t\in
\Delta^{*}$, and   $\omega$ denotes   the unique
singular Ricci-flat K\"{a}hler-Einstein metric   on $X_{0}$ with    $\omega\in
c_{1}(\mathcal{L})|_{X_{0}} \in H^{1}(X_{0}, \mathcal{PH}_{X_{0}})$,
 then
$$  F_{t}^{*}\omega_{t} \rightarrow  \omega, \  \ \ {\rm
when} \ \
 t \rightarrow 0,$$
 in the $C^{\infty}$-sense on  any compact subset $K \subset
X_{0,reg}$,  where $F_{t}: X_{0,reg} \rightarrow
X_{t}$ is a smooth  family of embeddings with $F_{0}={\rm Id}_{X_{0}}$.
Furthermore, the diameter of $ (X_{t}, \omega_{t})$, $t\in
\Delta^{*}$,   satisfies $${\rm
diam}_{\omega_{t}}(X_{t})\leq D,$$ where $D>0$ is a constant
independent of $t$.
\end{theorem}

This theorem shows a nice convergence behaviour of  the Ricci-flat  K\"{a}hler-Einstein   metrics $\omega_{t}$ away from the singular set of $X_{0}$.

Now we outline the proof of Theorem \ref{RZ1}.  Let  $\pi: \mathcal{X}\rightarrow \Delta$ be a Calabi-Yau degeneration with a  relative  ample line bundle  $\mathcal{L}$,  and  $\omega_{t}\in c_{1}(\mathcal{L})|_{X_{t}}\in H^{1,1}(X_{t}, \mathbb{R})$, $t\in
\Delta^{*}$,  be  the unique Ricci-flat  K\"{a}hler-Einstein   metric on $X_{t}$.  There is an embedding $\Phi_{\Delta}: \mathcal{X} \hookrightarrow \mathbb{CP}^{N}\times\Delta$ such that $\mathcal{L}^{m}\cong \Phi_{\Delta}^{*}\mathcal{O}_{\mathbb{CP}^{N}}(1)$ for integers  $m>0$ and $N>0$.    We denote $\omega_{t}^{o}=\frac{1}{m}\Phi_{\Delta}^{*}\omega_{FS}|_{X_{t}}$ for any $t\in\Delta^{*}$, where $\omega_{FS}$ is the Fubini-Study metric on $\mathbb{CP}^{N}$.   We obtained the following proposition in \cite{RoZ1}, which has some independent interests.

 \begin{proposition}\label{dia} Assume that the relative dualizing sheaf    is defined and trivial   $\varpi_{\mathcal{X}/ \Delta}\cong \mathcal{O}_{\mathcal{X}}$.
 The diameter of $\omega_{t}$ satisfies   $$ {\rm diam}_{\omega_{t}}(X_{t})\leq 2+D (-1)^{\frac{n^{2}}{2}}\int_{X_{t}}\Omega_{t}\wedge\overline{\Omega}_{t},$$ where  $\Omega_{t}$ is a relative holomorphic volume form, i.e. a no-where vanishing section of $\varpi_{\mathcal{X}/\Delta}$.
 \end{proposition}

  We  sketch  the proof of this estimate.
Note that $
\omega_{t}$ satisfies the Monge-Amp\`{e}re equation
$$ \omega_{t}^{n}=(-1)^{\frac{n^{2}}{2}}e^{\sigma_{t}}\Omega_{t}\wedge
\overline{\Omega}_{ t}, \ \ {\rm where} \ \
e^{\sigma_{t}}=V\left((-1)^{\frac{n^{2}}{2}}\int_{X_{t}}\Omega_{t}\wedge
\overline{\Omega}_{ t}\right)^{-1},
$$  where $V=n!{\rm Vol}_{\omega_{t}}(X_{t})$  is a
constant independent of $t$.
For a point  $p\in X_{0,reg}$, there are coordinates $z_{0}, \cdots,  z_{n}$ on  a
 neighborhood $U$ of $p$ in $ \mathcal{X}$ such that $t=\pi(z_{0}, \cdots,
 z_{n})=z_{0}$ and $p=(0, \cdots, 0)$. There is a $r_{0}>0$ such that $ \Delta \times
 \Delta^{n}\subset U$,  where $\Delta =\{|t|< r_{0}\}\subset
 \Delta$,  $\Delta^{n}=\{|z_{j}|< r_{0},  j =1, \cdots, n\}\subset \mathbb{C}^{n}$, and $\{t\}\times \Delta^{n} \subset
 X_{t} $.
   Note that locally
   $\omega_{t}^{o}$ and $\omega_{t}$ are families of K\"{a}hler
    metrics on $\Delta^{n}\subset \mathbb{C}^{n}$, and there is a constant $C'$ independent of
   $t$ such that
  $$C'^{-1}\omega_{E} \leq\omega_{t}^{o}\leq
   C' \omega_{E}, \ \
  $$  where $ \omega_{E}=\sqrt{-1}\partial\bar{\partial}\sum\limits_{i=1}^{n}
    |z_{i}|^{2}$ is the standard Euclidean K\"{a}hler metric  on $\Delta^{n} $.  Lemma 1.3 in
\cite{DPS} implies that   for any  $ \delta >0$, and any $t\in
\Delta \backslash\{0\}$, there is an
 open subset  $U_{t,\delta}$ of $\Delta^{n}$
  such that $$
{\rm Vol}_{\omega_{t}^{o}}(U_{t,\delta})\geq {\rm
Vol}_{\omega_{t}^{o}}(\Delta^{n})-\delta, \ {\rm  \ } \ {\rm
diam}_{\omega_{t}}(U_{t,\delta}) \leq
\hat{C}\delta^{-\frac{1}{2}},$$ where $\hat{C}$ is  a constant
independent of $t$.

   Let
$\delta_{t}=\frac{1}{2}{\rm Vol}_{\omega_{t}^{o}}(\Delta^{n})$, and let
$p_{t}\in U_{t,\delta_{t}}$. We get
$\delta_{t}\geq \frac{C}{2}{\rm Vol}_{\omega_{E}}(\Delta^{n})=\bar{\delta}$
and thus  $U_{t,\delta_{t}}\subset B_{\omega_{t}}(p_{t},r)
 $, where  $
r=\max \{ 1, 2\hat{C} \bar{\delta}^{- \frac{1}{2}}\}$. Since $U\subset
\mathcal{X}\backslash X_{0,sing}$, there is a constant $\kappa_{U}>0$ such
that
$$(-1)^{\frac{n^{2}}{2}}\Omega_{t}\wedge \overline{\Omega}_{ t}\geq
\kappa_{U}(\omega_{t}^{o})^{n}$$ on $U\cap X_{t}$.
We derive
 \begin{eqnarray*} {\rm Vol}_{\omega_{t}}(B_{\omega_{t}}(p_{t},r)) \geq
 {\rm Vol}_{\omega_{t}}(U_{t,\delta_{t}})  &
 = &
\frac{(-1)^{\frac{n^{2}}{2}}}{n!}e^{\sigma_{t}}\int_{U_{t,\delta_{t}}}\Omega_{t}\wedge
\overline{\Omega}_{ t}\\ & \geq &
\frac{\kappa_{U}e^{\sigma_{t}}}{n!}\int_{U_{t,\delta_{t}}} (\omega_{t}^{o})^{n}\\
&\geq & \frac{\kappa_{U}e^{\sigma_{t}}}{2}{\rm
Vol}_{\omega_{t}^{o}}(\Delta^{n})\\ &\geq & C e^{\sigma_{t}}{\rm
Vol}_{\omega_{E}}(\Delta^{n})= C e^{\sigma_{t}}.\end{eqnarray*}  The  Bishop-Gromov
  comparison theorem  shows that
$${\rm Vol}_{\omega_{t}}(B_{\omega_{t}}(p_{t},1))  \geq
\frac{1}{r^{2n}}{\rm
Vol}_{\omega_{t}}(B_{\omega_{t}}(p_{t},r))\geq
\frac{C}{r^{2n}}e^{\sigma_{t}}.
$$   By  Theorem 4.1 of Chapter 1 in \cite{SY2} or Lemma 2.3 in \cite{Pa},  we
obtain $$ {\rm diam}_{\omega_{t}}(X_{t})\leq 2+ 8n \frac{{\rm
Vol}_{\omega_{t}}(X_{t})}{{\rm
Vol}_{\omega_{t}}(B_{\omega_{t}}(p_{t},1))}\leq 2+
De^{-\sigma_{t}}, $$ and we obtain the estimate.

 Now  we return to the proof of Theorem   \ref{RZ1}.
  If we further assume that the central fiber $X_{0}$ is a Calabi-Yau variety, then it is shown in  Appendix B of \cite{RoZ1} that $(-1)^{\frac{n^{2}}{2}}\int_{X_{t}}\Omega_{t}\wedge\overline{\Omega}_{t}\leq C$ for a constant $C$ independent of $t$, and thus we obtain the diameter estimate in Theorem  \ref{RZ1}.

Let $\varphi_{t}$ be the unique potential function obtained by the Calabi-Yau theorem, which satisfies    $\omega_{t}=\omega_{t}^{o}+\sqrt{-1}\partial\overline{\partial}\varphi_{t}$, and the  Monge-Amp\`ere equation $$ (\omega_{t}^{o}+\sqrt{-1}\partial\overline{\partial }\varphi_{t})^{n}=(-1)^{\frac{n^{2}}{2}}e^{\sigma_{t}}\Omega_{t}\wedge\overline{\Omega}_{t}, \  \  \ \sup_{X_{t}} \varphi_{t}=0.   $$   The standard $C^{0}$-estimate  for the  Monge-Amp\`ere equation in \cite{Y3} gives $\|\varphi_{t}\|_{C^{0}}\leq C$ for a constant $C>0$ independent of $t$ by   reversing the roles of $\omega_{t}^{o}$ and $\omega_{t}$ (a trick originally coming from \cite{To1})    in  Section 3 of  \cite{RoZ1}.   Here we need  the priori   estimate of diameters to control the Sobolev constant.     Then by the $C^{2}$-estimate in \cite{Y3} and the Chern-Lu inequality, we have $$C'\omega_{t}^{o} \leq \omega_{t} \leq C_{K} \omega_{t}^{o},$$ for constants $C'>0$ and $C_{K}>0$ independent of $t$, on $X_{t}\cap K$, where $K$ is a compact subset of $\mathcal{X}\backslash X_{0,sing}$, and $C_{K}$ depends on $K$.  Furthermore, we obtain the $C^{2,\alpha}$-estimate $\|\varphi_{t}\|_{C^{2,\alpha}(K\cap X_{t})}\leq C_{2,K} $ by the Evans-Krylov theory (cf. \cite{Siu}),   and the higher order estimates $\|\varphi_{t}\|_{C^{l}(K\cap X_{t})}\leq C_{l,K} $ by the standard Schauder estimates (cf. \cite{GT}).

The smooth convergence of $\omega_{t}^{o}$ on  $X_{t}\cap K$ to $\omega_{0}^{o}$ on  $X_{0}\cap K$ implies that any sequence $\varphi_{t_{k}}$ has  subsequences convergence to a smooth bounded function $\varphi_{0}$ on $X_{0,reg}$, which satisfies that $ \omega_{0}= \omega_{0}^{o}+\sqrt{-1}\partial\overline{\partial}\varphi_{0}$ is a singular Ricci-flat K\"{a}hler-Einstein metric on $X_{0}$. By the uniqueness in Theorem \ref{EGZ}, $\omega_{0}=\omega$, and  we do not need to pass any sequence, and obtain the smooth convergence of $\varphi_{t}$ (respectively $\omega_{t}$) to $\varphi_{0}$ (respectively $\omega$) on $X_{0,reg}$.

Theorem \ref{RZ1} shows few information of the behaviour of $\omega_{t}$ approaching to the singularities $X_{0,sing}$ of $X_{0}$.   We do not expect the smooth convergence since the topology of the underlying manifold is changed, and thus we like to consider the Gromov-Hausdorff topology.

By the Gromov pre-compactness theorem
(Theorem \ref{G1}),  for any sequence $t_{k}$, a subsequence of $(X_{t_{k}}, \omega_{t_{k}})$ converges to a compact metric space $(Y,d_{Y})$ in the Gromov-Hausdorff sense. Since ${\rm Vol}_{\omega_{t}}(X_{t})= \frac{1}{n!}c_{1}^{n}(\mathcal{L}|_{X_{t}})\equiv {\rm const.}$,  Theorem \ref{CC} shows that    there is a closed subset $S\subset Y$ with
   Hausdorff dimension $\dim_{\mathcal{H}}S\leq 2n-4 $ such that $Y\backslash S
   $ is a Ricci-flat  K\"{a}hler $n$-manifold. By the smooth convergence of $\omega_{t}$ to $\omega$ on $X_{0,reg}$, we construct a local isometric embedding $\iota: (X_{0,reg}, \omega_{0})\hookrightarrow (Y\backslash S, d_{Y})$ in \cite{RoZ2}, and show that the Hausdorff co-dimension of $Y\backslash  \iota (X_{0,reg})$ is bigger or equal to $2$.  Then $\iota (X_{0,reg})$ is almost geodesic convex in $Y$ by \cite{CC2}, i.e. for any two points $y_{1}, y_{2}\in \iota (X_{0,reg})$, and any $\epsilon >0$, there is curve $\gamma \subset  \iota (X_{0,reg})$ such that $y_{1}=\gamma(0)$, $y_{2}=\gamma(1)$, and $${\rm length}_{d_{Y}}(\gamma)\leq d_{Y}(y_{1}, y_{2})+\epsilon.$$ Thus  $(Y, d_{Y}) $ is   the metric completion of
   $(X_{0,reg},\omega) $.

   In summary, we prove the following theorem in  \cite{RoZ2}.

\begin{theorem}[\cite{RoZ2}]\label{RZ2} Let $\pi: \mathcal{X}\rightarrow \Delta$, $X_{0}$, $\omega_{t}$ and $\omega$ be the same as in  Theorem \ref{RZ1}.
We have   $$(X_{t}, \omega_{t})
\stackrel{d_{GH}}\longrightarrow (Y, d_{Y}),  $$ when $t\rightarrow 0$, in the Gromov-Hausdorff sense, where  $(Y, d_{Y}) $ denotes   the metric completion of
   $(X_{0,reg},\omega) $,  which is a compact metric space.
\end{theorem}

Theorem \ref{RZ2} still offers little understanding of the metric behaviour near the singularities of $X_{0}$, and we expect more explicit asymptotic   behaviours.  Assume that  the Calabi-Yau variety $X_{0}$ in Theorem \ref{RZ1} and Theorem \ref{RZ2}   has only ordinary double points as singularities, i.e. $X_{0}$ is a conifold.    For any $p\in X_{0,sing}$, there is a neighborhood $U_{p}\subset \mathcal{X}$ such that $U_{p}\cap X_{t}$, $t\in\Delta$,  is  isomorphic to an open subset of  $$\mathcal{Q}_{t}=\{(z_{0}, \cdots z_{n})\in \mathbb{C}^{n+1} |z_{0}^{2}+\cdots +z_{n}^{2}=t\}.$$  Note that $\mathcal{Q}_{0}$ is the quadric cone, and for any $t\in\Delta^{*}$, $\mathcal{Q}_{t}$ is diffeomorphic to the total space of the cotangent bundle  $T^{*}S^{n}$ of $S^{n}$.  There is a complete Ricci-flat K\"ahler-Einstein metric
 \begin{equation}\label{eco-s}\omega_{co,t}=\sqrt{-1}\partial\overline{\partial }f_{t}(\rho) \end{equation}
 obtained in \cite{Ca} when $n=3$, and in  \cite{Sten} for the general dimension $n$, where $\rho=|z_{0}|^{2}+\cdots +|z_{n}|^{2}$, and $f_{t}(\rho) $ satisfies the ordinary differential equation $$\rho (f_{t}')^{n}+f_{t}''(f_{t}')^{n-1}(\rho^{2}-|t|^{2})=(\frac{n-1}{n})^{n+1} . $$ This equation can be solve explicitly by changing the  variable $\rho=|t |\cosh \tau$, and integrating $$\frac{d}{d\tau}(f_{t}'(\tau))^{n}=n|t|^{n-1}(\frac{n-1}{n})^{n+1} (\sinh \tau)^{n-1}.  $$
Note that $\omega_{co,0}=\omega_{co}=\sqrt{-1}\partial\overline{\partial} \rho^{1-\frac{1}{n}}$, the Ricci-flat cone metric, on $\mathcal{Q}_{0}$, and when $t\rightarrow 0$, $\omega_{co,t}$ converges smoothly to $\omega_{co}$.  More geometric properties of $\omega_{co,t}$, for example curvatures,  are  studied in Appendix A of \cite{FJY}.

  It is expected (cf. \cite{Wi} etc.) that $\omega_{co,t}$ approximates   $\omega_{t}$ of  Theorem \ref{RZ1} and Theorem \ref{RZ2} in a certain sense when $t\rightarrow 0$, i.e. for instance, $$ \|\omega_{co,t}-\omega_{t}\|_{C^{0}(\omega_{co,t},U_{p}\cap X_{t})}\leq C |t|^{\varepsilon},$$ for constants $C>0$ and $\varepsilon>0$.
Furthermore, $|t|^{-1}\omega_{t}$ should converge  smoothly to $\omega_{co,t_{0}}$ on $\mathcal{Q}_{t_{0}}$ for a $t_{0}\in \Delta^{*}$.  If they are true, there would be many interesting applications, for example the construction of special lagrangian spheres in $(X_{t}, \omega_{t})$ for $|t|\ll 1$.    Since the zero section $S^{n}\subset T^{*}S^{n}$ is a special lagrangian submanifold respecting to $\omega_{co,t}$ and certain holomorphic volume form on $\mathcal{Q}_{t}$, one can  deform   $S^{n}$ in $\mathcal{Q}_{t}$  to a  special lagrangian sphere in $X_{t}$ by the standard deformation technique  of special lagrangian submanifolds (cf. \cite{GHJ}).  As a matter of fact, this is one  goal of the study of metrics   along degenerations in \cite{Wi}.  The strategy to achieve such goal is the gluing method similar to those in Section 2.1. However, there are many difficulties even if the solution of Theorem \ref{EGZ} has the expected asymptotic behaviour (cf. \cite{Cha,Wi}).  Unfortunately, our new approach would not provide any new insight either.

\subsection{Gromov-Hausdorff convergence}
In this subsection, we continue our discussion in Section 2.2 by considering polarized Calabi-Yau  manifolds.

Let  $(X_{k},L_{k})$ be a sequence of polarized Calabi-Yau   manifolds, and  $\omega_{k}\in c_{1}(L_{k})$ be a sequence of Ricci-flat  K\"ahler-Einstien   metrics with  ${\rm diam}_{\omega_{k}}(X_{k})\leq D.$
The Gromov pre-compactness theorem
(Theorem \ref{G1}) asserts that    a subsequence of $(X_{k}, \omega_{k})$ converges to a compact metric space $(Y,d_{Y})$ in the Gromov-Hausdorff sense. Since ${\rm Vol}_{\omega_{k}}(X_{k})=\frac{1}{n!}c_{1}(L_{k})^{n}\geq \frac{1}{n!},$  $(X_{k}, \omega_{k})$ is non-collapsed, and
    Theorem \ref{CC} shows that    there is a closed subset $S\subset Y$  such that
   Hausdorff dimension $\dim_{\mathcal{H}}S\leq 2n-4 $, and  $Y\backslash S
   $ is a Ricci-flat  K\"{a}hler $n$-manifold.

   In \cite{DS}, Donaldson and Sun studied the algebro-geometric properties   of the limit $Y$, and showed that $Y$ is homeomorphic to a projective variety.

 \begin{theorem}[\cite{DS}]\label{D-S}      Let $(X_{k}, L_{k})$ be a sequence  of polarized Calabi-Yau manifolds of dimension $n$, and $\omega_{k}\in c_{1}(L_{k})$ be the unique Ricci-flat K\"ahler-Einstein metric.   We assume that  $${\rm Vol}_{\omega_{k}}(X_{k})=\frac{1}{n!}c_{1}(L_{k})^{n}\equiv\upsilon, \  \   {\rm diam}_{\omega_{k}}(X_{k}) \leqslant D$$ for  constants $D>0$ and $\upsilon>0$, and furthermore,  $$ (X_{k}, \omega_{k})  \stackrel{d_{GH}}\longrightarrow
  (Y, d_{Y})$$  in the Gromov-Hausdorff sense.  Then we have the follows.
    \begin{itemize}
  \item[i)] $Y$ is  homeomorphic to a Calabi-Yau variety $X_{\infty}$.
     \item[ii)]There are constants $m>0$ and $\bar{N}>0$ satisfying the following.  For any $k$,  there is an orthonormal basis $\Sigma_{k}$ of $H^{0}(X_{k}, L_{k}^{m})$ respecting to the $L^{2}$-norm induced by $\omega_{k}$, which induces an embedding $\Phi_{\Sigma_{k}}: X_{k} \hookrightarrow \mathbb{CP}^{\bar{N}}$ with $L_{k}^{m}= \Phi_{\Sigma_{k}}^{*} \mathcal{O}_{\mathbb{CP}^{\bar{N}}}(1)$. And $\Phi_{\Sigma_{k}}( X_{k})$ converges to $X_{\infty}$ in some  Hilbert schemes $\mathcal{H}il_{\bar{N}}^{P}$.
          \item[iii)]  The metric space structure on $Y$ is induced by the unique  singular Ricci-flat K\"{a}hler-Einstein metric $\omega \in \frac{1}{m} c_{1}( \mathcal{O}_{\mathbb{CP}^{\bar{N}}}(1)|_{X_{\infty}})$.
       \end{itemize}
   \end{theorem}

   By   Proposition 4.15 of  \cite{DS},   $Y$ is  homeomorphic to   a projective normal variety with only log-terminal singularities, denoted by $X_{\infty}$.  Note that the holomorphic volume forms $\Omega_{k}$ are parallel with respect to $\omega_{k}$, and converge  to a holomorphic volume form $\Omega_{\infty}$ on the regular locus $X_{\infty,reg}$ along the Gromov-Hausdorff convergence by normalizing $\Omega_{k}$ if  necessary. Thus the dualizing sheaf  $\varpi_{X_{\infty}}$ is trivial, i.e.  $\varpi_{X_{\infty}}\cong \mathcal{O}_{X_{\infty}}$, and $X_{\infty}$ is 1--Gorenstein.  Furthermore, the canonical divisor $\mathcal{K}_{X_{\infty}}$ is Cartier and trivial, which implies that $X_{\infty}$ has at worst  canonical singularities.  Then $X_{\infty}$ has only rational singularities,  $X_{\infty}$ is Cohen-Macaulay and is  Gorenstein.  Consequently,  $X_{\infty}$ is  a Calabi-Yau variety.

   We remark that Theorem \ref{D-S}  holds for polarized K\"ahler manifolds with bounded Ricci curvature (cf. \cite{DS}).  Readers are also  referred  to \cite{Tian} for the case of Fano manifolds, and \cite{DS,Tian} for the relevant history and applications.

     There are  some immediate interesting  applications of   Theorem \ref{D-S}.

First of all,  we apply  Theorem \ref{D-S} to the   finiteness  question.
 There is a conjecture due to S.-T. Yau, which says that there are only finite diffeomorphism  types of Calabi-Yau threefolds. There are many evidences to support it, and for example,  \cite{Gross} proves a weak version of this conjecture, i.e.,   up
to birational equivalence,
  elliptic fibred Calabi-Yau threefolds have only finite diffeomorphism  types.
 For any   constant   $D>0$,  we define a   set of polarized Calabi-Yau $n$-manifolds  $$\mathfrak{N}(n,  D)=\{(X,L)|  \omega\in c_{1}(L) \ {\rm with} \ {\rm Ric}(\omega)\equiv 0, \ \ {\rm diam}_{\omega}(X)\leq D \}.  $$  Note that $1\leq c_{1}(L)^{n}=n ! {\rm Vol}_{\omega}(X)\leq V(n,D)$ by the Bishop-Gromov comparison theorem.  Thus Theorem \ref{D-S} implies that  any $(X,L)\in \mathfrak{N}(n,  D)$ can be embedded into a comment   $\mathbb{CP}^{\bar{N}}$, and the image can be regarded as in finite  possible many  components of  Hilbert schemes.  Consequently, elements in  $ \mathfrak{N}(n,  D)$ have  only finite many possible complex  deformation classes, and of course, have  only finite many possible diffeomorphism types (cf. \cite{Bour}).    We will show the  generalization of   this finiteness result to   polarized projective K\"ahler manifolds  with Ricci curvature bounded from below in Appendix A.

Now, we apply  Theorem \ref{D-S} to the situation of Theorem \ref{RZ1} and  Theorem \ref{RZ2}, and obtain that the compact metric space  $Y$  in Theorem \ref{RZ2} is homeomorphic to a Calabi-Yau variety $X_{\infty}$.  Actually, it is easy to see  that  $X_{\infty}$ is isomorphic to  $X_{0}$ of Theorem \ref{RZ1} (cf.  Lemma 2.2 of \cite{zha1}). However it is not trivial since the embeddings in  Theorem \ref{D-S} are induced by orthonormal basis, which are highly transcendental, and would not coincide with algebraic embedding induced by the relative ample line bundle $\mathcal{L}$ from the hypothesis.
 One corollary  is the uniqueness of the filling-in  for degenerations of Calabi-Yau manifolds.

    \begin{corollary}\label{coro}  Let  $(\pi: \mathcal{X}\rightarrow \Delta, \mathcal{L})$ and $(\pi': \mathcal{X}'\rightarrow \Delta, \mathcal{L}')$ be  two degenerations of polarized Calabi-Yau manifolds with  Calabi-Yau  varieties   $X_{0}$ and $X_{0}'$ as the central fibers respectively.  If  there is a sequence of points  $t_{k} \rightarrow 0$ in $\Delta$, and there is a sequence of isomorphism $\psi_{k}: X_{t_{k}}\rightarrow X_{t_{k}}'$ such that $\psi_{k}^{*}\mathcal{L}|_{X_{t_{k}}}\cong  \mathcal{L}'|_{X_{t_{k}}'}$, then $X_{0}$ is isomorphic to $X_{0}'$.
\end{corollary}

Note that $\mathcal{X}$ may not be birational to $\mathcal{X}'$  in this corollary.   If we have a stronger assumption that $\mathcal{X}\backslash X_{0}$ is isomorphic to $\mathcal{X}'\backslash X_{0}'$, then the conclusion   is a direct  consequence of   \cite[Theorem 2.1]{Bo} and  \cite[Corollary 4.3]{Od}.

 We finish this section by recalling some aspects of  the proof of  Theorem \ref{D-S}.
Let  $(X_{k}, L_{k})$,  $\omega_{k}$  and $(Y,d_{Y})$ be the same as in  Theorem \ref{D-S}.
The main step in  the proof of Theorem \ref{D-S} is to construct uniformly many peak sections of $L_{k}^{\tilde{\nu}}$ on $X_{k}$, for a $\tilde{\nu} \in\mathbb{N}$,  which are also sufficiently   many, and  enough to induce embeddings for not only all of $X_{k}$, but also the limit $Y$. A peak section $s$ of $L_{k}^{\tilde{\nu}}$ means that for a point $p\in X_{k}$, $|s|_{k}(p)\geq 1$,  $\|s\|_{k,L^{2}}^{2}=\int_{X_{k}}|s|_{k}^{2}\sim 1 $, and $|s|_{k}\ll 1$ on the  complement of a small neighborhood of $p$, where $|\cdot |_{k}$ denotes the Hermitian metric on $L_{k}^{\tilde{\nu}}$ with curvature $ \tilde{\nu} \omega_{k}$.  The peak sections are constructed by gluing a local model first, and then using the H\"{o}rmander $L^{2}$-estimate.

  The local model used in \cite{DS} is as following.     For any $y\in Y$, the tangent cone $Y_{y}$ at $y$ is a metric cone $C(M_{y})$ over a compact metric space $M_{y}$, and the regular locus $Y_{y,reg}$ is an open dense subset of $Y_{y}$.  Let $o\in Y_{y}$ be the distinguished vertex point, and let  $r(q)={\rm dist}(q,o)$, i.e. the distance of $o$ to the point $q\in  Y_{y}$.  On $Y_{y,reg}$, the metric is  a Ricci-flat K\"ahler-Einstein cone metric $dr^{2}+r^{2}g_{M}$, where $g_{M}$ is a Sasakian-Einstein metric on the regular locus $M_{y,reg}$.  It is standard that the Ricci-flat K\"ahler-Einstein metric can be written as $\omega_{y}=\frac{\sqrt{-1}}{2}\partial\overline{\partial}r^{2}$ (cf. \cite{Spa}).  If $L_{o}$  is the smooth trivial line bundle on $Y_{y,reg}$, i.e. $L_{o}\simeq Y_{y,reg}\times \mathbb{C}$ in the sense of  smooth bundles, then we define a $U(1)$-connection on $L_{o}$  $$A_{o}=\frac{1}{4}(\overline{\partial}r^{2}-\partial r^{2}), \  \ {\rm whose \ \ curvature} \ \ F_{A_{o}}=-\sqrt{-1}\omega_{y}.    $$ Note that $\overline{\partial}_{A_{o}}=\overline{\partial}+A_{o}^{0,1}$ is a Cauchy-Riemann  operator, i.e. $\overline{\partial}_{A_{o}}^{2}=0$, and $\overline{\partial}_{A_{o}}$  induces a holomorphic structure on $L_{o}$.  If we define  $$\sigma_{o}=e^{-\frac{r^{2}}{4}},  \  \  \ {\rm then} \ \  \ \overline{\partial}_{A_{o}} \sigma_{o}=\overline{\partial} \sigma_{o}+A_{o}^{0,1}\sigma_{o}=0,$$ i.e.
   $\sigma_{o}$ is a holomorphic section of $(L_{o},A_{o})$.  The local model is $(L_{o}, A_{o}, \sigma_{o})$.

   We recall the gluing argument of \cite{DS}, and assume that $M_{y}$ is smooth, i.e. $o$ is an  isolated singular point,   and $H^{1}(M_{y}, \mathbb{Z})$ and $H^{2}(M_{y}, \mathbb{Z})$ are torsion free for simplicity.  Denote  $U=\{q\in C(M_{y})| \delta \leq r(q) \leq R\}$ for constants $0<\delta\ll 1$ and $R\gg 1$, and we identify $  M_{y}=\{q\in C(M_{y})|  r(q)=1\}$.  For an $\epsilon>0$, let $\omega$ be a K\"{a}hler metric respecting to a complex structure $J$ on $U$ such that $$\|\omega -\omega_{y}\|_{C^{0}}\leq \epsilon, \  \ {\rm  and }  \  \  \|J-J_{y}\|_{C^{0}}\leq \epsilon,$$ where  $J_{y}$ is the complex structure on $Y_{y,reg}$, and let $A$ be a $U(1)$-connection on a complex line bundle $L$ with curvature  $F_{A}=-\sqrt{-1}\omega$, which implies that $\overline{\partial}_{J,A}$ is a Cauchy-Riemann operator, and  induces a holomorphic structure on $L$.  Since $\omega $ (respectively $\omega_{y}$) represents  $ c_{1}(L)$ (respectively $ c_{1}(L_{o})$), a integral class, and $\omega $ is close to $\omega_{y} $, we have $ c_{1}(L)= c_{1}(L_{o})=0$, and $L$ is isomorphic to $L_{o}$ as smooth bundles. Hence we regard $\sigma_{o}$ as a section of $L$.

  For a certain cut off function $\beta$,  we like to prove that $(\omega, J, L, A, \sigma= \beta \sigma_{o})$ satisfies the so called Property (H) in \cite{DS}, which guarantees that the result section  after applying the  H\"{o}rmander $L^{2}$-estimate to these data is a peak section.
 Property (H) reads that
   \begin{itemize}
  \item[i)]$\| \sigma\|_{L^{2}}<(2\pi)^{\frac{n}{2}}$,  \item[ii)] $|\sigma|(p)>\frac{3}{4}$ for a point $p\in U' \subset U$,
     \item[iii)] for any smooth section $\tau$ of $L$, $|\tau|(p)\leq C(\|\overline{\partial}_{J,A} \tau\|_{L^{p}(U')}+\| \tau\|_{L^{2}(U')}),$
     \item[iv)] $\|\overline{\partial}_{J,A}\sigma\|_{L^{2}}\leq \min \{\frac{1}{8C\sqrt{2}}, \frac{(2\pi)^{\frac{n}{2}}}{10\sqrt{2}}\}$
          \item[v)]  $\|\overline{\partial}_{J,A} \sigma\|_{L^{p}(U')}<\frac{1}{8C}$.
          \end{itemize}
  Property (H) can be obtained by direct calculations similar to those in Section 2 of \cite{Don}, if we further  have $\|A -A_{o}\|_{C^{0}}\ll 1$. However this condition may not be true due to the non-trivialness of the first Betti number of $U$.

Note that  $\alpha=(A -A_{o})|_{M_{y}}$ is a connection on $L\otimes L_{o}^{-1}|_{M_{y}}$ with curvature $F_{\alpha}=-\sqrt{-1}(\omega-\omega_{y})|_{M_{y}}$.   Since $\alpha$ is transformed to $\alpha+d \log u$ under a  gauge change $u:M_{y} \rightarrow M_{y}\times U(1)$,  we have $\tilde{\alpha}=\alpha_{H}+\alpha'$ by the Hodge theory  where $\alpha_{H}$ is a harmonic 1-form, and $\alpha'$ satisfies $d^{*}\alpha'=0$, $d\alpha'=F_{\alpha}$ and  $ \|\alpha'\|_{L^{2}}\leq C\|F_{\alpha}\|_{L^{2}}$.  The harmonic part $\alpha_{H}$ is a flat connection, and can be regarded as in $H^{1}(M_{y}, U(1))$, the moduli space of flat connections on $M_{y}$.  Note that $H^{1}(M_{y}, U(1))=\mathcal{H}^{1}(M_{y})/H^{1}(M_{y},\mathbb{Z}) $, and the $L^{2}$-norm on the space $\mathcal{H}^{1}(M_{y})$ of harmonic 1-forms induces a metric on $H^{1}(M_{y}, U(1))$.  In \cite{DS}, it is shown that there is  an open subset $W\in H^{1}(M_{y}, U(1))$, for example we  say  $W=\{\zeta \in \mathcal{H}^{1}(M_{y}) |\|\zeta\|_{L^{2}}<\epsilon_{W}\ll 1\}/H^{1}(M_{y},\mathbb{Z})$, such that  if $\alpha_{H}\in W$, then one can extend the gauge change $u$ to a gauge change $\tilde{u}$ on $U$ such that $\tilde{A}= A+d \log \tilde{u}$ satisfies   $\|\tilde{A} -A_{o}\|_{C^{0}}\leq C' (\epsilon_{W}+\|F_{A}-F_{A_{o}}\|_{L^{2}})$, and
 furthermore,  $(\omega, J, L, \tilde{A}, \sigma)$ satisfies the  Property (H).

Now Dirichlet's theorem asserts that there is an $m\in \mathbb{N}$ satisfying  that for any $\zeta \in H^{1}(M_{y}, U(1))$, there is a $1\leq \nu \leq m$ such that $\nu \zeta=\zeta' {\rm mod}\{ H^{1}(M_{y},\mathbb{Z})\}$ for a $\zeta'\in W$.   If the above $\alpha_{H}$ does not belong to $ W$, then we define $\tau (\nu A) \in W$ such that  $\nu \alpha_{H}=\tau (\nu A ){\rm mod}\{ H^{1}(M_{y},\mathbb{Z})\}  $  for a $\nu \leq m$.
 Let  $\tilde{U}= \{q\in C(M_{y})| m^{-\frac{1}{2}}\delta \leq r(q) \leq  R\}$, and $\mu_{\nu}: U\rightarrow \tilde{U}$ be the dilation map given by  $\mu_{\nu}(x,r)=(x, \nu^{-\frac{1}{2}}r)$ for any $(x,r)\in C(M_{y})$ where $x\in M_{y}$ and $r=r((x,r))$.

   By the convergence of $\omega_{k}$ and the definition of tangent cone, there are embeddings $\Psi_{k}: \tilde{U} \rightarrow X_{k}$ such that  $\|m_{0} \Psi_{k}^{*}\omega_{k} -\omega_{y}\|_{C^{0}}\leq \epsilon m^{-1}, $  for an  $m_{0}\in \mathbb{Z} $,    and  $  \|\Psi_{k}^{*}J_{k}-J_{y}\|_{C^{0}}\leq \epsilon m^{-1},$ where  $J_{k}$ is the complex structure of $X_{k}$.
 Let $A_{k}$ be the $U(1)$-connection on $L_{k}$ with curvature $F_{A_{k}}=-\sqrt{-1}\omega_{k}$ and compatible with the  holomorphic structure of $L_{k}$.
 By passing to a subsequence, Dirichlet's theorem shows as above that there is a  $1\leq \nu \leq m$ such that $\tau (\nu m_{0}  A_{k}) \in W$ is defined as  $\|\nu m_{0} \Psi_{k}^{*}  F_{A_{k}}-\nu F_{A_{o}}\|_{C^{0}}\leq \epsilon$. Then $\nu m_{0} A_{k}$ is gauge equivalent to a $U(1)$-connection $\tilde{A}_{k}$ of $\Psi_{k}^{-1}L_{k}^{\nu m_{0}}$ on $\tilde{U}$, and
  $$\|\tilde{A}_{k} -\nu A_{o}\|_{C^{0}}\leq C' (\epsilon_{W}+\|\nu m_{0}  \Psi_{k}^{*}  F_{A_{k}}-\nu F_{A_{o}}\|_{L^{2}})\leq C' (\epsilon_{W}+ \epsilon). $$ Since $\nu\mu_{\nu}^{*}\omega_{y}=\omega_{y}$, $ \mu_{\nu}^{*}J_{y}=J_{y}$ and  $\nu\mu_{\nu}^{*}A_{o}=A_{o}$,  one obtain that \\  $(\mu_{\nu}^{*}\Psi_{k}^{*}\nu m_{0}\omega_{k}, \mu_{\nu}^{*}\Psi_{k}^{*}J_{k}, \mu_{\nu}^{-1}\Psi_{k}^{-1}L_{k}^{\nu m_{0}}, \mu_{\nu}^{*}\tilde{A}_{k}, \sigma)$ satisfies the  Property (H).

By  identifying  $U$ with $ \Psi_{k}(\mu_{\nu}(U))$,  the  H\"{o}rmander $L^{2}$-estimate shows that there is a smooth  section $\tilde{\sigma}$ of $L_{k}^{\nu m_{0}}$ such that
  $$ \overline{\partial}_{ J_{k}, \nu m_{0} A_{k}}(\tilde{\sigma}) = \overline{\partial}_{J_{k}, \nu m_{0} A_{k}}(\sigma), \  \  {\rm with} \  \  \| \tilde{\sigma}\|_{L^{2}} \leqslant    \frac{1}{\sqrt{\nu m_{0}}}\|  \overline{\partial}_{J_{k}, \nu m_{0} A_{k}}(\sigma)\|_{L^{2}}.  $$ The peak section is  $s=\sigma-\tilde{\sigma}$.

   \subsection{Equivalence}
    Theorem \ref{RZ2} shows that under the algebro-geometric assumption of Calabi-Yau degeneration with  Calabi-Yau central fiber,  Ricci-flat  K\"{a}hler-Einstein   metrics are non-collapsed, and converge to a  compact metric space in the Gromov-Hausdorff sense. On the other hand,  Theorem \ref{D-S} says that under the differential  geometric assumption of Gromov-Hausdorff convergence, we have the limit being a Calabi-Yau variety.  It is speculated that   degenerating to  Calabi-Yau varieties should be equivalent to the  non-collapsing  convergence, which is confirmed by a recent work   \cite{Ta}.   Actually, we have a triple equivalence
    among not only the degeneration to Calabi-Yau varieties, and the  Gromov-Hausdorff  convergence, but also  the finiteness of  Weil-Petersson distance.

     Let $\pi: \mathcal{X}\rightarrow \Delta$ be a flat family of  Calabi-Yau
$n$-manifolds, and  $\mathcal{L}$ be   a relative  ample line bundle on $\mathcal{X}$.   The variation of   Hodge structures  gives a  natural   semi-positive form, a possibly  degenerated   K\"{a}hler metric,  on $\Delta$,
called the Weil-Petersson metric, $$ \omega_{WP}=-\frac{\sqrt{-1}}{2\pi}\partial\overline{\partial}\log \int_{X_{t}}(-1)^{\frac{n^{2}}{2}} \Omega_{t} \wedge \overline{\Omega}_{t}\geq0,$$ where  $\Omega_{t} $ is a relative holomorphic volume form, i.e. a no-where vanishing  section of $ \varpi_{\mathcal{X}/\Delta}$   (cf. \cite{Ti}).    The Weil-Petersson metric $ \omega_{WP}$  is the curvature of the first Hodge bundle  $\mathcal{H}^{n}=R^{n}\pi_{*}\mathbb{C}\otimes \mathcal{O}_{\Delta}$  with a natural  Hermitian metric, and describes the deformation of complex structures.

In \cite{CGH}, Candelas,  Green and H\"{u}bsch found some nodal degenerations of Calabi-Yau 3-folds with finite Weil-Petersson distance.  In general,   \cite{Wang1}  shows that  if $(\pi: \mathcal{X}\rightarrow \Delta, \mathcal{L})$ is    a degeneration of polarized  Calabi-Yau manifolds, and if the central fiber   $X_{0}$ is a Calabi-Yau variety, then the Weil-Petersson distance between $\{0\}$ and the interior $\Delta^{*}$ is finite, i.e.  $\omega_{WP}$ is not complete on $\Delta^{*}$.   Conversely,   if we assume that the Weil-Petersson distance  of $\{0\}$  is finite,  then after a finite base change   $\pi: \mathcal{X}\rightarrow \Delta$ is birational to a degeneration  $\pi': \mathcal{X}'\rightarrow \Delta$ such that $\mathcal{X}\backslash X_{0}\cong \mathcal{X}'\backslash X_{0}'$, and $X_{0}'$ is a Calabi-Yau variety by a  paper   \cite{To}.  As a consequence,  the  algebro-geometric degenerating  Calabi-Yau manifolds to a Calabi-Yau variety is equivalent  to the finiteness of the Weil-Petersson distance. In \cite{Ta},    the further equivalence to the  Gromov-Hausdorff  convergence  is also  established.
 In summary, we have the following theorem.

 \begin{theorem}[\cite{Wang1,To,Ta}]\label{Ta-To}
 Let $(\pi: \mathcal{X}\rightarrow \Delta, \mathcal{L})$ be a  degeneration of polarized Calabi-Yau manifolds.
 Then the  following statements are equivalent.
 \begin{itemize}
  \item[i)] After a finite base change,   $\pi: \mathcal{X}\rightarrow \Delta$ is birational to a degeneration  $\pi': \mathcal{X}'\rightarrow \Delta$ such that $\mathcal{X}\backslash X_{0}\cong \mathcal{X}'\backslash X_{0}'$, and $X_{0}'$ is a Calabi-Yau variety.
   \item[ii)]  $$(X_{t}, \omega_{t})
\stackrel{d_{GH}}\longrightarrow (Y, d_{Y}),  $$ when $t\rightarrow 0$, in the Gromov-Hausdorff sense, where  $\omega_{t}$ denotes the unique Ricci-flat  K\"{a}hler-Einstein   metric
in  $
c_{1}(\mathcal{L})|_{X_{t}}\in H^{1,1}(X_{t}, \mathbb{R})$, $t\in
\Delta^{*}$, and   $(Y, d_{Y}) $ is a compact metric space.
 \item[iii)]  $\{0\}$ has finite Weil-Petersson distance to the interior $\Delta^{*}$.
       \end{itemize}
\end{theorem}

An application of such equivalence is the completion of  the moduli space of polarized Calabi-Yau manifolds, which is surveyed in the next section.

\section{Moduli space of polarized Calabi-Yau manifolds}
Let $\mathcal{M}^{P}$ be the moduli space of polarized Calabi-Yau manifolds $(X,L)$ of dimension $n$ with a fixed Hilbert polynomial $P=P(\mu)=\chi (X,L^{\mu})$, i.e.  $$ \mathcal{M}^{P}=\big\{(X,L)|P(\mu)=\chi (X,L^{\mu}) \big\}/ \sim,$$ where $(X_{1},L_{1})\sim(X_{2},L_{2})$ if and only if  there is an isomorphism $\psi: X_{1}  \rightarrow X_{2}$ such that $L_{1}=\psi^{*}L_{2}$.  We denote the equivalent class $[X,L]\in  \mathcal{M}^{P}$ represented by $(X,L)$.

The  Bogomolov-Tian-Todorov's  unobstructedness  theorem  of Calabi-Yau manifolds implies that $\mathcal{M}^{P}$ is a complex orbifold (cf. \cite{Ti,Tod}).  The variation of   Hodge structures  gives a natural orbifold  K\"{a}hler metric on $\mathcal{M}^{P}$,
called the Weil-Petersson metric,   which  is the curvature of the first Hodge bundle    with a natural  Hermitian metric  (cf. \cite{Ti}).
From the algebro-geometric viewpoint,    Viehweg proved   in \cite{Vie}  that
$\mathcal{M}^{P}$ is a  quasi-projective variety, and coarsely represents the moduli functor  $\mathfrak{M}^{P}$ for polarized Calabi-Yau manifolds with   Hilbert polynomial $P$.

We like   to understand $\mathcal{M}^{P}$ by  considering  Ricci-flat K\"{a}hler-Einstein metrics.  The Calabi-Yau theorem gives  a continuous  map$$\label{map0} \mathcal{CY}: \mathcal{M}^{P}  \rightarrow \mathcal{M}et,  \  \  {\rm by}  \  \  \left[X,L\right]  \mapsto  (X, \omega),$$
 where $\omega$ is the unique Ricci-flat K\"ahler-Einstein metric representing $c_{1}(L)$. However, $ \mathcal{CY}$ is not injective in general  since $\mathcal{M}^{P}$ contains the information of complex structures.

  The  compactifications of   moduli spaces were   studied in various cases, for example, the Mumford's compactification of moduli spaces for curves (cf.  \cite{M-GIT}),  the Satake compactification of moduli spaces for Abelian varieties  (cf. \cite{Sat}), and more  recently the compact moduli spaces  for general type  stable varieties of higher dimension  (cf. \cite{Kol1}).   Because of the importance of Calabi-Yau manifolds in mathematics and physics, it is  also  desirable to have  compactifications of $\mathcal{M}^{P}$.

 For constructing compactifications of $\mathcal{M}^{P}$, Yau suggested   that one uses the Weil-Petersson metric to obtain a metric  completion of   $\mathcal{M}^{P}$   first, and then tries  to compactify this completion (cf. \cite{Lu}). In \cite{Wang2}, an alternative approach   is proposed by using the Gromov-Hausdorff distance,   instead of the Weil-Petersson metric,   to construct a completion.  In a recent paper \cite{zha1}, we  constructed  a completion  of  $\mathcal{M}^{P}$ via Ricci-flat K\"{a}hler-Einstein metrics and Gromov-Hausdorff topology, which can be viewed as a  partial compactification.

  Let  $\overline{\mathcal{CY}( \mathcal{M}^{P} )}$ be the  closure of the image  $\mathcal{CY}( \mathcal{M}^{P} ) $ in $\mathcal{M}et$.  There is a natural metric space structure on $\overline{\mathcal{CY}( \mathcal{M}^{P} )}$ by restricting the Gromov-Hausdorff distance.  Since $(\mathcal{M}et, d_{GH})$ is complete, $\overline{\mathcal{CY}( \mathcal{M}^{P} )}$ is the completion in the Gromov-Hausdorff sense.  However, we do not expect that $\overline{\mathcal{CY}( \mathcal{M}^{P} )}$ have some  algebro-geometric properties because of the non-injectivity  of $\mathcal{CY}$.  We obtain an enlarge moduli space $\overline{\mathcal{M}}^{P}$  such that $\mathcal{CY}$  extends to surjection from $\overline{\mathcal{M}}^{P}$ to $\overline{\mathcal{CY}( \mathcal{M}^{P} )}$ in \cite{zha1}.

 \begin{theorem}\label{main} There is a Hausdorff topological space $\overline{\mathcal{M}}^{P}$, and    a surjection  $$\overline{\mathcal{CY}}:\overline{\mathcal{M}}^{P} \rightarrow   \overline{\mathcal{CY}( \mathcal{M}^{P} )}$$    satisfying  the follows.    \begin{itemize}
  \item[i)]   $\mathcal{M}^{P}$ is  an open dense subset of  $ \overline{\mathcal{M}}^{P}$,   and $\overline{\mathcal{CY}}|_{\mathcal{M}^{P}}=\mathcal{CY}$.
   \item[ii)]  For any $p\in  \overline{\mathcal{M}}^{P}$,   $\overline{\mathcal{CY}}(p)$ is homeomorphic to a Calabi-Yau variety.
     \item[iii)]  There is an exhaustion  $$ \mathcal{M}^{P}\subset  \mathcal{M}_{m(1)} \subset  \mathcal{M}_{m(2)}  \subset  \cdots \subset \mathcal{M}_{m(l)}   \subset \cdots \subset \overline{\mathcal{M}}^{P}=\bigcup_{l\in\mathbb{N}}  \mathcal{M}_{m(l)},$$ where $m(l)\in\mathbb{N}$ for any $l\in\mathbb{N}$,  such that   $\mathcal{M}_{m(l)} $ is a  quasi-projective variety, and there is an ample line bundle $\lambda_{m(l)}$ on   $\mathcal{M}_{m(l)} $.
           \item[iv)]  Let  $(\pi: \mathcal{X}\rightarrow \Delta, \mathcal{L})$ be a degeneration of polarized  Calabi-Yau manifolds  with   a Calabi-Yau variety $X_{0}$ as the central fiber.  Assume that for any $t\in\Delta^{*}$, there is an ample line bundle $L_{t}$ on $X_{t}$ such that $L_{t}^{k}\cong \mathcal{L}|_{X_{t}}$ for a  $k\in\mathbb{N}$, and $[X_{t},L_{t}]\in \mathcal{M}^{P}$.
      Then there is    a unique morphism $\rho: \Delta \rightarrow \mathcal{M}_{m(l)} $,  for  $l\gg 1$,  such that  $\overline{\mathcal{CY}}(\rho(t))$ is homeomorphic to  $ X_{t}$ for any $t\in\Delta$, and $$ \overline{\mathcal{CY}}(\rho(t)) \rightarrow \overline{\mathcal{CY}}(\rho(0)),$$ when $t\rightarrow 0$, in the Gromov-Hausdorff sense.
      Furthermore,  $\rho^{*}\lambda_{m(l)}=\pi_{*} \varpi_{\mathcal{X}/\Delta}^{\nu(l)}$ for a  $\nu(l)\in\mathbb{N}$.
\end{itemize}
\end{theorem}

When $n=2$, a Calabi-Yau variety is a K3 orbifold, and a degeneration of K3 surfaces to a K3 orbifold is called a degeneration of type I.  It is well-known that one can fill the holes in the moduli space of K\"{a}hler  polarized K3 surfaces by some K\"{a}hler  K3 orbifolds, and  obtain  a complete moduli space
     (cf. \cite{Kob1,KT}).
     The relationship between such moduli space and the degeneration of Ricci-flat K\"{a}hler-Einstein metrics  is also established    in \cite{Kob1,KT}.

We recall the construction of  $\overline{\mathcal{M}}^{P}$.
For any $D>0$, we define a subset $\mathcal{M}^{P}(D)$  of  $\mathcal{M}^{P}$ by  $$\mathcal{M}^{P}(D)=\{ \left[X,L\right] \in \mathcal{M}^{P}| \  {\rm   Ricci-flat \  metric} \   \omega \in c_{1}(L)  \ {\rm   with   \  diam}_{\omega}(X)\leq D \}. $$
 We have that if $D_{1}\leq D_{2}$, then $\mathcal{M}^{P}(D_{1})\subset \mathcal{M}^{P}(D_{2})$, and $\mathcal{M}^{P}=\bigcup\limits_{D>0}\mathcal{M}^{P}(D) $.
Note that for  a sequence  $\left[X_{k}, L_{k}\right]\in \mathcal{M}^{P}(D)$, if $(X_{k}, \omega_{k})$   converges to a compact metric space $Y$ in the Gromov-Hausdorff sense, then by Theorem \ref{D-S}, there are embeddings   $\Phi_{k}: X_{k} \hookrightarrow \mathbb{CP}^{N}$ for an  $N>0$ independent of $k$ such that  $L_{k}^{m}\cong \Phi_{k}^{*} \mathcal{O}_{\mathbb{CP}^{N}}(1)$ for an $m>0$, and
    $\Phi_{k}(X_{k})$ converges to
  a Calabi-Yau variety $X_{\infty}$ in the Hilbert scheme  $\mathcal{H}ilb^{P_{m}}_{N}$, which is  homeomorphic to $Y$.  By Matsusaka's Big Theorem (cf. \cite{Mat}), we take $m$ large enough such that    for any $[X,L]\in\mathcal{M}^{P} $,
      $L^{m}$ is very ample, and $H^{i}(X,L^{m})=\{0\}$, $i>0$.   Thus  we have an embedding $\Phi_{\Sigma}: X
 \hookrightarrow \mathbb{CP}^{N}$, and $\Phi_{\Sigma}(X)\in \mathcal{H}il_{N}^{P_{m}}$.

 Let $\pi_{\mathcal{H}}: \mathcal{U}_{N}\rightarrow \mathcal{H}ilb^{P_{m}}_{N}$ be the universal family over the Hilbert scheme $ \mathcal{H}ilb^{P_{m}}_{N}$ of the  Hilbert polynomial $P_{m}(\mu)=P(m\mu)$, and $\mathcal{H}_{N}^{o}\subset (\mathcal{H}ilb^{P_{m}}_{N})_{red}$ be the Zariski  open  subset    parameterizing    smooth varieties.  Let
    $\mathcal{H}_{N}$ be the subset of $ (\mathcal{H}ilb^{P_{m}}_{N})_{red}$ such that  $\mathcal{H}_{N}^{o}  \subset \mathcal{H}_{N} \subset \overline{ \mathcal{H}_{N}^{o}}$  where $\overline{ \mathcal{H}_{N}^{o}}$ denotes
      the Zariski closure  of   $\mathcal{H}_{N}^{o}$ in $ (\mathcal{H}ilb^{P_{m}}_{N})_{red}$, and  a  point   $p\in \mathcal{H}_{N}$ if   $X_{p}=\pi_{\mathcal{H}}^{-1}(p)$ is a Calabi-Yau variety.

Recall that for any Calabi-Yau variety $X_{p}$ where $p\in\mathcal{H}_{N}$,  the Ricci-flat
 K\"{a}hler-Einstein  metric $\omega\in c_{1}(\mathcal{O}_{\mathbb{CP}^{N}}(1)|_{X_{p}})$ induces  an  $SU(N+1)$-orbit $ RO(X_{p},\mathcal{O}_{\mathbb{CP}^{N}}(1)|_{X_{p}})\subset  \mathcal{H}ilb^{P_{m}}_{N}$  by (\ref{eor}).
 If we let  $$ \mathcal{R}_{N}= \bigcup_{p\in\mathcal{H}_{N}} RO(X_{p},\mathcal{O}_{\mathbb{CP}^{N}}(1)|_{X_{p}})\subset \mathcal{H}_{N},$$ then there is a natural $SU(N+1)$-action on $ \mathcal{R}_{N}$, which is induced by the $SL(N+1)$-action on $ \mathcal{H}ilb^{P_{m}}_{N}$. Let $$\mathcal{M}_{m}=  \mathcal{R}_{N}/ SU(N+1).$$  Note that the reduced Hilbert scheme $ (\mathcal{H}ilb^{P_{m}}_{N})_{red}$ is Hausdorff under the analytic topology, and so is   the subset  $\mathcal{R}_{N}$. Thus the quotient $\mathcal{M}_{m}$ by the   compact Lie group  $ SU(N+1)$ is also Hausdorff.
  It is clear that both $\mathcal{M}^{P}(D)$ and $\mathcal{M}^{P}$ are subsets of $\mathcal{M}_{m}$, and furthermore,  if $X$ is a Calabi-Yau variety homeomorphic to a compact metric space in $\overline{\mathcal{CY}(\mathcal{M}^{P}(D))} $, then  $X\in \mathcal{M}_{m}$.  We obtain the enlarged moduli space $\overline{\mathcal{M}}^{P}=\cup \mathcal{M}_{m}$ by takeing a increasing sequence of $D_{j}$.

However, it is unclear whether $\mathcal{M}_{m}$ has any algebro-geometric property  from the above construction. We need the full strength  of \cite{Vie} to prove that   $\mathcal{M}_{m}$ is a quasi-projective variety.  Firstly, we show that $\mathcal{H}_{N}$ can be given  an open subscheme structure in Section 3.2 of \cite{zha1} by the deformation of varieties with canonical singularities \cite{Ka}.  Secondly,   by \cite[Theorem 2.1]{Bo} and  \cite[Corollary 4.3]{Od},  if   $( \mathcal{X}_{1}\rightarrow S, \mathcal{L}_{1})$ and $( \mathcal{X}_{2}\rightarrow S, \mathcal{L}_{2})$ are     two flat families of polarized  Calabi-Yau varieties  over a germ of smooth curve  $(S,0)$, then any isomorphism of these two families over $S\backslash\{0\}$ extends to an isomorphism over $S$, which implies the   separateness condition of  $\mathcal{M}_{m}$.  Finally, the construction of Section 8 in \cite{Vie} shows that there is  a geometric quotient $ \mathcal{H}_{N} \rightarrow \mathcal{M}_{m}'$ of the natural  $SL(N+1)$-action on $\mathcal{H}_{N}$,  and an ample line bundle $\lambda_{m}$ on  $\mathcal{M}_{m}' $, which implies that $\mathcal{M}_{m}'$ is a quasi-projective variety.
     If  $O(X_{p},\mathcal{O}_{\mathbb{CP}^{N}}(1)|_{X_{p}})$ is the $SL(N+1)$-orbit  of the $SL(N+1)$-action for a $p\in\mathcal{H}_{N}$, then
     the uniqueness of the K\"ahler-Einstien metric  $\omega$ implies that
 $$ RO(X_{p},\mathcal{O}_{\mathbb{CP}^{N}}(1)|_{X_{p}})= O(X_{p},\mathcal{O}_{\mathbb{CP}^{N}}(1)|_{X_{p}})\bigcap \mathcal{R}_{N}. $$ Thus  $$  \mathcal{M}_{m}=  \mathcal{R}_{N}/ SU(N+1)= \mathcal{H}_{N}/ SL(N+1)$$ with the quotient topology induced by the analytic topology of $\mathcal{H}_{N}$, and   is homeomorphic to the underlying variety of $ \mathcal{M}_{m}'$. We obtain the conclusion.

     K\"ahler-Einstein metrics were previously   used to construct compactifications of moduli spaces  for K\"ahler-Einstein orbifolds in \cite{Ma}, and it is proved that such compactification  coincides with the standard Mumford's  compactification in the case of  curves.   As a matter of fact, the idea of the above construction is from \cite{Ma}.  The real slice $ \mathcal{R}_{N}$ induced by K\"ahler-Einstein metrics is an analogue of the zero set of momentum map in the Geometric Invariant Theory (cf. \cite{M-GIT,Th}).

     In \cite{zha1}, we also proved  that the  points  in $\overline{\mathcal{M}}^{P}\backslash   \mathcal{M}^{P}$ have finite Weil-Petersson distances, which is a  corollary of Theorem
\ref{Ta-To}.

       \begin{theorem}\label{main2} Let   $\overline{\mathcal{M}}^{P}$ and   $\overline{\mathcal{CY}}$  be the same as in Theorem \ref{main}.
          \begin{itemize}
  \item[i)]   For any point  $x\in  \overline{\mathcal{M}}^{P}\backslash   \mathcal{M}^{P}$,  there is a curve $\gamma$ such that $\gamma(0)=x$,  $\gamma( (0,1] )\subset  \mathcal{M}^{P}$ and the length of $\gamma$  under the Weil-Petersson metric  $\omega_{WP}$ is finite, i.e.   $${\rm length}_{\omega_{WP}}(\gamma) < \infty .$$
    \item[ii)]  Let  $(\pi: \mathcal{X}\rightarrow \Delta, \mathcal{L})$ be a degeneration of polarized  Calabi-Yau manifolds   such that  for any $t\in\Delta^{*}$, $L_{t}^{k}\cong \mathcal{L}|_{X_{t}}$ for a  $k\in\mathbb{N}$, and $[X_{t},L_{t}]\in \mathcal{M}^{P}$, where $L_{t}$ is an ample line bundle.   If the Weil-Petersson distance between  $0\in \Delta$ and the interior  $\Delta^{*}$ is finite,
      then there is    a unique morphism $\varrho: \Delta \rightarrow \mathcal{M}_{m(l)} $,  for  $l\gg 1$,  such that $\mathcal{CY}(\varrho(t))$ is homeomorphic to  $X_{t}$,  $t\in\Delta^{*}$.
   \end{itemize}
\end{theorem}

 A simple concrete example for Theorem \ref{main} and Theorem \ref{main2} is the mirror Calabi-Yau threefold of the quintic threefold constructed in \cite{COGP} (cf. Section 18 in  \cite{GHJ}), i.e.   $X_{t}$ is the crepant resolution of the quotient  $$Y_{s}= \{\left[z_{0}, \cdots, z_{4}\right]\in \mathbb{CP}^{4} | z_{0}^{5}+ \cdots + z_{4}^{5}+s z_{0} \cdots z_{4}=0\}/(\mathbb{Z}_{5}^{5}/\mathbb{Z}_{5})$$ of the quintic by $\mathbb{Z}_{5}^{5}/\mathbb{Z}_{5}$,  where $s^{5}= t\in \mathbb{C}$.    By choosing a polarization, $ \mathcal{M}^{P}= \mathbb{C}\backslash \{ 1\}$, and $0$ is an orbifold point of $ \mathcal{M}^{P}$.   When $t=1$,    $X_{1}$  is a Calabi-Yau variety with only one   ordinary double points, and    however, $t=\infty$ is a large complex limit point, which implies that  $t=\infty$ is the cusp end of  $ \mathcal{M}^{P}$   and   has infinite Weil-Petersson distance.
 Thus  $\overline{\mathcal{M}}^{P}=\mathbb{C}$.

\section{Surgeries}

Extremal transitions and flops are algebro-geometric surgeries providing
ways to connect two topologically distinct projective manifolds,
which are interesting in both mathematics and physics. In the minimal
model program, all smooth minimal models  in a birational
equivalence class are connected by  sequences  of flops. The famous Reid's fantasy conjectures that all Calabi-Yau
threefolds are connected to each other by extremal transitions, so as to form a huge connected
web.  The goal of this section is to study the behaviour of Ricci-flat K\"ahler-Einstien metrics under these surgeries, extremal transitions and flops.

\subsection{Degeneration of K\"{a}hler classes}
In this subsection, we study the limit behaviour of Ricci-flat K\"{a}hler-Einstein metrics when their K\"{a}hler classes approach to  the boundary of the K\"{a}hler cone on a fixed Calabi-Yau manifold, which is parallel to the  case of degenerations.

As in the degeneration case, we have  a prior estimate for the diameter of   Ricci-flat K\"{a}hler-Einstein metric   obtained in \cite{To1} and \cite{zha3} (See also Theorem 3.1 in  \cite{RuZ}).  More precisely, if $X$ is a Calabi-Yau $n$-manifold, and $\omega_{0}$ is a Ricci-flat K\"{a}hler-Einstein metric on $X$, then for any Ricci-flat K\"{a}hler-Einstein metric $\omega$ on $X$, we have $$ {\rm diam}_{\omega}(X)\leq 32n+ D(\int_{X}\omega \wedge \omega_{0}^{n-1})^{n},$$ where $D>0$ is a constant   depending on $\omega_{0}$, but independent   of $\omega$.

Let $X$ be a Calabi-Yau variety, and $\bar{\pi}:\bar{X}\rightarrow X$ be a crepant resolution, i.e. a resolution such that $\bar{\pi}^{*}\varpi_{X}=\varpi_{\bar{X}}$, which implies that $\bar{X}$ is a Calabi-Yau manifold. If $L$ is an ample line bundle on $X$, then $\bar{\pi}^{*}L$ is semi-ample and big, but no longer ample, i.e. $\bar{\pi}^{*}c_{1}(L)\in \partial\overline{\mathbb{K}_{\bar{X}}}=\overline{\mathbb{K}_{\bar{X}}}\backslash \mathbb{K}_{\bar{X}}$.
  The following convergence theorem is proved
  in \cite{To}.

 \begin{theorem}[\cite{To}]\label{To1}
 \label{To}   Let $X$ be a Calabi-Yau variety, and $\bar{\pi}:\bar{X}\rightarrow X$ be a crepant resolution, and  $L$ be an ample line bundle on $X$.
 For any family of K\"{a}hler classes
 $\alpha_{s}\in \mathbb{K}_{\bar{X}}$, $s\in (0,1]$, with $\lim\limits_{s\rightarrow 0}\alpha_{s}=\bar{\pi}^{*}c_{1}(L)$, if  $\bar{\omega}_{s}\in \alpha_{s} $ is the  Ricci-flat K\"{a}hler-Einstein metric, then
  $\bar{\omega}_{s}$ converges smoothly  to
 $\bar{\pi}^{*}\omega$ on any compact subset of $
 \bar{\pi}^{-1}(X_{reg})$, when $s\rightarrow 0$,   where   $ \omega$ is the unique singular  Ricci-flat  K\"{a}hler-Einstein metric
 representing $c_{1}(L)$ on $X$.
 \end{theorem}

Note that in the  above theorem,    the diameter of  $\bar{\omega}_{s}$ is uniformly bounded, i.e. $${\rm diam}_{\bar{\omega}_{s}}(\bar{X})\leq D,$$ for constant $D>0$ independent of $s$.
  For any sequence $s_{k}\rightarrow 0$,
  the Gromov pre-compactness theorem
(Theorem \ref{G1}) says that    a subsequence of $(\bar{X}, \bar{\omega}_{s_{k}})$ converges to a compact metric space $(Y,d_{Y})$ in the Gromov-Hausdorff sense, and   Theorem \ref{CC} shows that    there is a closed subset $S\subset Y$ with
   Hausdorff dimension $\dim_{\mathcal{H}}S\leq 2n-4 $ such that $Y\backslash S
   $ is a Ricci-flat  K\"{a}hler $n$-manifold.   In \cite{RoZ1}, it is proved that $Y$ is the metric completion of $(X_{reg}, \omega)$.
    An analogue  of Theorem \ref{D-S} is obtained for the  case of $X$ being  a conifold in \cite{Song1}, and later  is  generalized   to   any Calabi-Yau variety  in \cite{Song2}. Furthermore,   the analogue for more general K\"{a}hler metrics is also studied  in a recent  preprint \cite{LZT},

\begin{theorem}[\cite{Song2}]\label{song} Let $X$, $\bar{X}$ and $\bar{\omega}_{s}$ be the same as in Theorem \ref{To}.
If $Y$ is the Gromov-Hausdorff  limit a sequence $(\bar{X}, \bar{\omega}_{s_{k}})$, then
   $Y$  is homeomorphic to $X$.
\end{theorem}

As in the case of degeneration, we also like to know more explicit  asymptotical behaviour when metrics are approaching to  the singularities.    If $X$ is a K3 orbifold, Kobayashi's theorem in Section 2.1 gives a very satisfactory     description.  However it is not clear in the higher dimensional case.

  We assume that  the Calabi-Yau variety $X$ in Theorem \ref{To} and Theorem \ref{song}   has dimension 3,  and  only ordinary double points as singularities, i.e. for any $p\in X_{sing}$, there is a neighborhood $U_{p}\subset X$  isomorphic to an open subset of  the quadric  $\mathcal{Q}=\{(w_{1}, \cdots, w_{4})\in \mathbb{C}^{4} |w_{1}w_{2}=w_{3}w_{4}\}$.  The crepant resolution $\bar{X}$ is locally given by the small resolution of  the quadric cone  $$\bar{\mathcal{Q}}=\{(w_{1}, \cdots, w_{4},[y_{0},y_{1}])\in \mathbb{C}^{4}\times \mathbb{CP}^{1} |w_{1}y_{1}=w_{3}y_{0},  y_{0}w_{2}=y_{1}w_{4}\}\rightarrow  \mathcal{Q} ,$$ where $ \bar{\mathcal{Q}}$ is biholomorphic to the total space of the bundle $\mathcal{O}_{\mathbb{CP}^{1}}(1)\oplus\mathcal{O}_{\mathbb{CP}^{1}}(1)$.  For any $s>0$,
    a complete Ricci-flat K\"ahler-Einstein metric $\bar{\omega}_{co, s}$ is construced in \cite{Ca} by
 \begin{equation}\label{eco-r}\bar{\omega}_{co, s}=\sqrt{-1}\partial\overline{\partial }(\bar{f}_{s}(\rho)+4s^{2} \log (|y_{0}|^{2}+|y_{1}|^{2})),  \end{equation} where $\rho=|w_{1}|^{2}+ \cdots + | w_{4}|^{2}$, and $\bar{f}_{s}(\rho)$ satisfies the following ordinary differential equation $$(\rho \bar{f}_{s}'(\rho))^{3}+6s^{2} (\rho \bar{f}_{s}'(\rho))^{2}- \rho^{2}=0, $$
 which can be solved explicitly.  When $s \rightarrow 0$, $\bar{\omega}_{co, s}$ converges smoothly to the Ricci-flat cone metric $\omega_{co}$ on $\mathcal{Q}$ given by  (\ref{eco}).  Like the case of degeneration,  we expect that those metrics $\bar{\omega}_{s}$ on $\bar{X}$ are asymptotic to $\bar{\omega}_{co, s}$ near singularities when $s \rightarrow 0$.

 Now we consider the flops, which connect distinct Calabi-Yau manifolds in the same birational class.
If $X$ is a Calabi-Yau variety  admitting  two different crepant
 resolutions $ (\bar{X}_{1}, \bar{\pi}_{1})$ and
 $(\bar{X}_{2}, \bar{\pi}_{2})$ with both exceptional subvarieties of codimension  at least  2, the process of
going from $\bar{X}_{1}$ to  $\bar{X}_{2}$ is called  a flop,
denoted by $$\bar{X}_{1} \rightarrow X \dashrightarrow
\bar{X}_{2}.$$   The simplest case is the local  flop for the quadric $\mathcal{Q}$ of dimension 3. Note that $\mathcal{Q}$ has two  crepant resolutions, one is $ \bar{\mathcal{Q}}$ as above and the other is $$\bar{\mathcal{Q}}'=\{(w_{1}, \cdots, w_{4},[y_{0},y_{1}])\in \mathbb{C}^{4}\times \mathbb{CP}^{1} |w_{1}y_{1}=w_{4}y_{0},  y_{0}w_{2}=y_{1}w_{3}\}\rightarrow  \mathcal{Q} ,$$
i.e.   we have a flop $$\bar{\mathcal{Q}}\rightarrow \mathcal{Q}  \dashrightarrow
\bar{\mathcal{Q}}'.$$  Flops were used by Tian and Yau in \cite{TY-flop} for constructing topological distinct Calabi-Yau threefolds.

 The following theorem is proved in \cite{RoZ1}, which asserts that flops are continuous in the Gromov-Hausdorff sense when one considers Ricci-flat K\"{a}hler-Einstein metrics.

 \begin{theorem}[\cite{RoZ1}]\label{t1.01+} Let  $X$ be an $n$-dimensional Calabi-Yau
variety,  and $L$ be an ample
line bundle on $X$.  Assume that $X$ admits two crepant resolutions
$(\bar{X}_{1}, \bar{\pi}_{1})$ and $(\bar{X}_{2}, \bar{\pi}_{2})$.
 Let
 $\bar{\omega}_{1,s}$ (resp. $\bar{\omega}_{2,s}$),  $s\in (0, 1] $,  be a family of Ricci-flat
  K\"{a}hler metrics on $\bar{X}_{1} $ (resp. $\bar{X}_{2} $) with  K\"{a}hler classes
  $\lim\limits_{s\rightarrow 0}[\bar{\omega}_{\alpha,s}]=\bar{\pi}_{\alpha}^{*} c_{1}(L)
  $, $\alpha=1,2$.
  Then there exists a compact metric space $(Y, d_{Y})$ such
  that $$( \bar{X}_{1}, \bar{\omega}_{1,s}) \stackrel{d_{GH}}\longrightarrow (Y,
d_{Y}) \stackrel{d_{GH}}\longleftarrow ( \bar{X}_{2},
\bar{\omega}_{2,s}),   \ \ \ \  when \ \  s\rightarrow 0.
$$
 Furthermore, $(Y, d_{Y})$ is isometric to the metric completion
   $(X_{reg},\omega) $ where $\omega \in c_{1}(L)$ is the singular  Ricci-flat
  K\"{a}hler-Einstein  metric on $X$
\end{theorem}

In the minimal
   model program, it is proved  that  for any two Calabi-Yau
   $n$-manifolds $X$ and $X'$  birational to each other, there are sequences  of
   flops connecting  $X$ and $X'$ in  \cite{Kawamata}, which  was obtained previously  in  \cite{Kol} for the case of dimension 3,  i.e., there is  a sequence of
varieties $ X_{1}, \cdots, X_{k}$ such that $X=X_{1}$, $X'=X_{k}$,
and $X_{j+1}$ is obtained by a flop from $X_{j}$. Consequently there
are normal  projective varieties $ X_{0,1}, \cdots, X_{0,k-1}$, and
small resolutions $\bar{\pi}_{j}:X_{j}\rightarrow X_{0,j}$ and
$\bar{\pi}_{j}^{+}:X_{j}\rightarrow X_{0,j-1}$.  Hence $$X_{1}\rightarrow X_{0,1} \dashrightarrow X_{2}\rightarrow  X_{0,2} \dashrightarrow \cdots \rightarrow X_{0,k-1} \dashrightarrow X_{k}.  $$  Now we restrict  to the case of dimension $3$, i.e. $\dim_{\mathbb{C}}X=\dim_{\mathbb{C}}X'=3$.
 By  \cite{Kol},
$X_{j}$ has  the same singularities as $X$, and thus $X_{j}$ is
smooth.  Since the exceptional locus of $\bar{\pi}_{j}$ and
$\bar{\pi}_{j}^{+}$ are of co-dimension at least   2,  $X_{0,j}$ has
only canonical singularities, and  the canonical bundle  of
$X_{0,j}$ is  trivial (cf. Corollary 1.5 in  \cite{Kaw}).  Therefore
$X_{0,j}$ is a three-dimensional Calabi-Yau variety, and $X_{j}$ is a
three-dimensional Calabi-Yau
manifold.

 For a fixed  Calabi-Yau manifold $X$, if we denote $$\mathcal{CY}(X,\mathbb{K}_{X})=\{(X, \omega)|{\rm all \ of} \ \omega \ {\rm with} \   {\rm  Ric}(\omega)\equiv0,   \ {\rm Vol}_{\omega}(X)=1 \}\subset (\mathcal{M}et, d_{GH}),$$ then
by  Theorem \ref{t1.01+}, for any $j>0$, there is a compact metric space $(Y,d_{Y})$ belonging to both of $\overline{\mathcal{CY}(X_{j},\mathbb{K}_{X_{j}})}$ and $
\overline{\mathcal{CY}(X_{j+1},\mathbb{K}_{X_{j+1}})}$, where
$\overline{\mathcal{CY}(X_{j},\mathbb{K}_{X_{j}})}$ denotes the closure of
$\mathcal{CY}(X_{j},\mathbb{K}_{X_{j}})\subset (\mathcal{M}et, d_{GH})$,   and furthermore $Y$ is homeomorphic to $X_{0,j}$ by Theorem \ref{song}.
Thus
$$\overline{\mathcal{CY}(X_{j},\mathbb{K}_{X_{j}})}\bigcup
\overline{\mathcal{CY}(X_{j+1},\mathbb{K}_{X_{j+1}})}$$ is path connected.  We obtain the following corollary.

    \begin{corollary}[\cite{RoZ1}]\label{c1.2}  If $X$ is a three-dimensional Calabi-Yau manifold, then
   $$\bigcup_{
  all \ such  \ X' }\{\overline{\mathcal{CY}(X',\mathbb{K}_{X'})}|\ a \  Calabi-Yau \ manifold \ X' \  birational \ to \ X\}
 $$ is path connected in $(\mathcal{M}et, d_{GH})$.
 \end{corollary}

 We expect that this corollary still holds for higher dimensional Calabi-Yau manifolds. One of the difficulties to obtain a  proof is the  lack of   knowledge of singular  Ricci-flat
  K\"{a}hler-Einstein  metrics on singular  Calabi-Yau varieties. More precisely, if a Calabi-Yau variety $X$ does not   admit a  crepant resolutions,  or there is no Calabi-Yau degeneration with $X$ as central fiber, we do not know  whether a singular  Ricci-flat
  K\"{a}hler-Einstein  metric on $X$ induces a metric structure on $X$.

 \subsection{Extremal transition}
 This subsection study extremal transitions.

 Let $X$ be a singular Calabi-Yau variety.   Assume that $X$  admits a crepant resolution $\bar{\pi}: \bar{X}\rightarrow X$, and   there is a Calabi-Yau degeneration $\pi: \mathcal{X}\rightarrow \Delta$ with  $X_{0}=X$. The process of
going from $\bar{X}$ to $X_{t}$, $t\in \Delta^{*}$,  is called an
extremal transition, denoted by $$\bar{X}\rightarrow X
\rightsquigarrow X_{t}.$$ We call this process a  conifold
transition if $X$ is a conifold, i.e. $X$ has only ordinary double points as singularities. Locally a conifold transition exists in the case of dimension 3, i.e.  $$\bar{\mathcal{Q}}\rightarrow \mathcal{Q}_{0}
\rightsquigarrow \mathcal{Q}_{t},$$ where $\mathcal{Q}_{t}=\{(z_{0}, \cdots, z_{3})\in \mathbb{C}^{4}| z_{0}^{2}+ \cdots + z_{3}^{2}=t\}$, $t\in \Delta$, is the quadric,  and $\bar{\mathcal{Q}}$ is the small resolution of $\mathcal{Q}_{0} $.   It contracts the $\mathbb{CP}^{1}$ in $\mathcal{O}_{\mathbb{CP}^{1}}(1)\oplus\mathcal{O}_{\mathbb{CP}^{1}}(1)$ to obtain the 3-dimensional quadric  cone $\mathcal{Q}_{0}$ first, and then deforms the complex structure of $\mathcal{Q}_{0}$ to get those $\mathcal{Q}_{t}$, which are diffeomorphic to the total space of the  cotangent bundle $T^{*}S^{3}$.
A extremal  transition
  is an  algebro-geometric
 surgery
  providing   ways to connect two  topologically distinct Calabi-Yau
manifolds, which increases complex moduli and decreases K\"{a}hler moduli.

  The famous Reid's
fantasy
 conjectures  that all Calabi-Yau threefolds are connected to
each other by   extremal transitions, possibly including non-K\"{a}hler
Calabi-Yau threefolds,   so as to form a huge connected web (cf.
\cite{Re,Ro}).  There is also  a  projective version of this
conjecture,  the connectedness conjecture for  moduli spaces for
Calabi-Yau threefolds (cf. \cite{Gro,Gro0,Ro}).
In physics, extremal transitions   are related
to the vacuum degeneracy problem  in string theory (cf. \cite{CGH,CGGK,GMS,GH,Ro}).  The explanation of  extremal  transition in physics is that type II string theories modeled by topologically distinct Calabi-Yau vacua can be continuously connected via suitable black hole condensations (cf. \cite{Stro,GMS}). Readers are
referred to the survey article \cite{Ro} for topology, algebraic geometry, and
 physics properties of extremal transitions.

 Note that physicists are interested in Calabi-Yau manifolds in the first place, because   the holomony group of a Ricci-flat
  K\"{a}hler-Einstein  metric is  a subgroup of $SU(n)$ despite   of many works using only algebro-geometric setting in the mathematical physics literatures.  It makes perfect sense to consider extremal transitions  with metric.

Physicists  P. Candelas and X. C. de la Ossa conjectured  in \cite{Ca}  that
extremal  transitions   should be
 "continuous in the space of Ricci-flat K\"{a}hler-Einstein  metrics", even though   these processes  involve  topologically
distinct Calabi-Yau manifolds. This conjecture is verified  in
\cite{Ca} for the local conifold transition, i.e. $$\bar{\omega}_{co, s}\rightarrow \omega_{co} \leftarrow \omega_{co, t},$$
 when $s\rightarrow 0$ and $t\rightarrow0$, in the $C^{\infty}$-sense, where $\omega_{co}=\sqrt{-1}\partial\overline{\partial}\rho^{\frac{2}{3}} $ is the Ricci-flat
  K\"{a}hler-Einstein  cone metric on $\mathcal{Q}_{0}$, and $\bar{\omega}_{co, s}$ (respectively $\omega_{co, t}$) is the metric given by (\ref{eco-r}) (respectively (\ref{eco-s})). In this case, we can regard that the local  conifold transition is a surgery with metric, which  shrinks a holomorphic $\mathbb{CP}^{1}$ to one point,   obtains a singular Ricci-flat space, and   then  replaces the singular  point by  a special  lagrangian sphere $S^{3}$.

 In the general case, we obtain the following theorem in \cite{RoZ1}, which can be obtained by combining Theorem \ref{RZ1} and Theorem \ref{To1}.

 \begin{theorem}[\cite{RoZ1}]\label{t1.1}Let  $X$ be a  Calabi-Yau
$n$-variety. Assume that
\begin{itemize}
  \item[i)]  there is a  degeneration $(\pi: \mathcal{X}\rightarrow \Delta, \mathcal{L})$ of polarized Calabi-Yau manifolds  such that $X_{0}=X$.
  For any $t\in \Delta^{*}$, let $\omega_{t}$ be the
unique Ricci-flat
  K\"{a}hler-Einstein  metric on $X_{t} $ representing  $  c_{1}(\mathcal{L})|_{X_{t}}
  $.\item[ii)] $X$ admits a crepant  resolution $\bar{\pi}: \bar{X}\rightarrow X$.
  Let
  $\bar{\omega}_{s}$, $s\in (0, 1] $,  be  a family of Ricci-flat
  K\"{a}hler-Einstein  metrics with  K\"{a}hler classes $\lim\limits_{s\rightarrow 0}[\bar{\omega}_{s}]=\bar{\pi}^{*} c_{1}(\mathcal{L})|_{X} $
   in $H^{1,1}(\bar{X}, \mathbb{R} )$.
   \end{itemize}
  Then there exists a compact  metric space $(Y, d_{Y})$ such
  that $$(X_{t}, \omega_{t}) \stackrel{d_{GH}}\longrightarrow
(Y, d_{Y}) \stackrel{d_{GH}}\longleftarrow  (\bar{X}, \bar{\omega}_{s}), \  \  \ {\rm when} \ t\rightarrow 0, \ s\rightarrow 0.
$$
   Furthermore, $(Y, d_{Y})$ is isometric to the metric completion
   $(X_{reg},\omega) $ where $\omega \in c_{1}(\mathcal{L})|_{X}$ is the singular  Ricci-flat
  K\"{a}hler-Einstein  metric on $X$.
 \end{theorem}

 By Theorem \ref{D-S}, $Y$ is homeomorphic to the Calabi-Yau variety  $X$.

 The following is a simple  example that   Theorem
\ref{t1.1} can apply.  Let $\bar{X}$ be the complete intersection in
$\mathbb{CP}^{4}\times \mathbb{CP}^{1} $ given by
$$y_{0}\mathfrak{g}(z_{0},\cdots ,z_{4})+y_{1}\mathfrak{h}(z_{0},\cdots,z_{4})=0,
\ \ \ y_{0}z_{4} -y_{1}z_{3}=0, $$ where $z_{0},\cdots,z_{4}$ are
homogeneous coordinates of $\mathbb{CP}^{4}$, $y_{0},y_{1}$ are
homogeneous coordinates of $\mathbb{CP}^{1}$, and $\mathfrak{g}$ and
$\mathfrak{h}$ are generic  homogeneous polynomials of degree 4.
Then $\bar{X}$ is a  crepant resolution of the quintic conifold
$X$ given by $$z_{3}\mathfrak{g}(z_{0},\cdots
,z_{4})+z_{4}\mathfrak{h}(z_{0},\cdots,z_{4})=0. $$
Hence there is a conifold transition $\bar{X}\rightarrow X
\rightsquigarrow \tilde{X} $ for any smooth quintic $\tilde{X}$ in
$\mathbb{CP}^{4}$.  Theorem  \ref{t1.1} implies that there is a
family of Ricci-flat K\"{a}hler metrics
 $\bar{\omega}_{s}$, $s\in (0, 1] $,  on $\bar{X} $ and a family of Ricci-flat  smooth quintic
 $(X_{t}, \omega_{t})$, $ t\in\Delta^{*}$,   such that
 $X_{1} =\tilde{X}$, and $$(X_{t}, \omega_{t}) \stackrel{d_{GH}}\longrightarrow
(Y, d_{Y}) \stackrel{d_{GH}}\longleftarrow  (\bar{X}, \bar{\omega}_{s}),
$$ for a compact metric space $(Y, d_{Y})$ homeomorphic to $X$.

As an application of Theorem \ref{t1.1} and Theorem \ref{t1.01+}, we
shall explore the path connectedness
 properties of certain class of  Ricci-flat Calabi-Yau threefolds.
Inspired by string theory in physics, some physicists made a
 projective version  of Reid's fantasy (cf.  \cite{CGH,GH,Ro}), the so called connectedness conjecture,  which is
formulated more precisely  in \cite{Gro} (See also \cite{Gro0}).
This conjecture says that there is a huge connected  web $ \Gamma$
such that   nodes of $ \Gamma$ consist of   all deformation classes
of Calabi-Yau threefolds, and two nodes are connected
$\mathfrak{D}_{1}- \mathfrak{D}_{2}$ if $\mathfrak{D}_{1}$ and $
\mathfrak{D}_{2}$ are related by an extremal transition, i.e.,  there
is a Calabi-Yau 3-variety $X$ that  admits a crepant resolution
$\bar{X}\in \mathfrak{D}_{1}$,  and there is a Calabi-Yau degeneration $\pi:\mathcal{X}\rightarrow \Delta$ such that $X=X_{0}$, and
$X_{t}\in \mathfrak{D}_{2}$ for any $ t\in
\Delta^{*}$.  It is shown in \cite{GH,CGGK,ACJM,Gro} that many known  Calabi-Yau threefolds are
connected to each other in the above sense.
 By combining the connectedness conjecture
and  Theorem \ref{t1.1} and Theorem \ref{t1.01+}, we reach a metric
version of connectedness conjecture as follows: if $
\mathfrak{CY}_{3}$ denotes  the set of Ricci-flat Calabi-Yau
threefolds $(X,\omega) $ with volume 1, then  the closure $
\overline{\mathfrak{CY}}_{3}$ of $ \mathfrak{CY}_{3}$ in
$(\mathcal{M}et,d_{GH}) $ is path connected, i.e.,  for any two
points  $p_{1}$ and $p_{2}\in \overline{\mathfrak{CY}}_{3}$, there
is a  path
$$\gamma:[0,1] \longrightarrow \overline{\mathfrak{CY}}_{3}\subset
(\mathcal{M}et,d_{GH})$$ such that $ p_{1}=\gamma(0)$ and
$p_{2}=\gamma(1)$.

 Given  a class   of Calabi-Yau 3-manifolds known  to be connected by
extremal transitions  in algebraic geometry, Theorem
\ref{t1.1}  can be used to show that the
closure of the  class of  Calabi-Yau 3-manifolds
   is path connected in $(\mathcal{M}et,d_{GH}) $.   In
   \cite{GH}, it is shown  that all complete  intersection  Calabi-Yau
   manifolds  (CICY) of dimension 3 in  products  of
  projective
 spaces  are connected by conifold  transitions. Furthermore,
   in  \cite{ACJM,CGGK}   a large number of complete
intersection  Calabi-Yau
   3-manifolds in toric varieties  are  verified  to be
   connected
   by extremal transitions, which include
   Calabi-Yau hypersurfaces in all  toric manifolds obtained by resolving
     weighted projective 4-spaces.   As a corollary  of Theorem \ref{t1.1},  we
 obtain the following result.

 \begin{corollary}[\cite{RoZ1}]\label{c1.3}
  \begin{itemize}
  \item[i)]   $$\overline{\bigcup_{
  all \ such   \  X
 } \{\mathcal{CY}(X,\mathbb{K}_{X})|CICY \ threefold   \  X  \ in \ products \ of \
  projective \
 spaces\}}
 $$  is path connected in $(\mathcal{M}et,d_{GH}) $.
 \item[ii)]  There is a path connected subset  of $\overline{\mathfrak{CY}}_{3}$, which contains all of CICY  threefolds    in  products  of
  projective
 spaces, and
    Calabi-Yau hypersurfaces  in
  toric 4-manifolds    obtained by resolving a  weighted projective
 4-space.
 \end{itemize}
 \end{corollary}

Finally, we remark that  many of the above results are expected to have some analogues for $G_{2}$ and $Spin(7)$ manifolds.  Like   Calabi-Yau threefolds in the string theory,
 $G_{2}$ and $Spin(7)$ manifolds play some roles in the M-theory (cf. \cite{Guk}).  There are also topological changing processes for  $G_{2}$ and $Spin(7)$ manifolds, for example the $G_{2}$-flop transition (see \cite{Guk,CHNP}),  which is an analogue of extremal transition for Calabi-Yau threefolds. Theorem \ref{t1.1} provides a trivial example for this situation. Let $X$, $X_{t}$, $\bar{X}_{s}$, $\omega_{t}$ and $\bar{\omega}_{s}$ be the same as in Theorem \ref{t1.1} with $\dim_{\mathbb{C}}X=3 $, and $g_{t}$ (respectively $\bar{g}_{s}$) be the corresponding Riemannian metric of  $\omega_{t}$ (respectively $\bar{\omega}_{s}$). It is standard that the holonomy groups  of  $\tilde{g}_{t}=g_{t}+d \theta^{2}$ and  $\tilde{\bar{g}}_{s}=\bar{g}_{s}+d \theta^{2}$ belong to  $G_{2}$ (cf. \cite{Joy}). If we denote $\tilde{X}_{t}=X_{t}\times S^{1}$ and $\tilde{\bar{X}}=\bar{X}\times S^{1}$,  Theorem \ref{t1.1} implies that
$$(\tilde{X}_{t}, \tilde{g}_{t}) \stackrel{d_{GH}}\longrightarrow
(\tilde{Y}, d_{\tilde{Y}}) \stackrel{d_{GH}}\longleftarrow  (\tilde{\bar{X}}, \tilde{\bar{g}}_{s}), \  \  \ {\rm when} \ t\rightarrow 0, \ s\rightarrow 0,
$$ where $\tilde{Y}$ is homeomorphic to $X \times S^{1}$, and $d_{\tilde{Y}}$ is induced by a $G_{2}$-metric on $X_{reg} \times S^{1}$.
However since neither advanced PDE nor algebraic geometry is available in this case, it is difficult to obtain a  general result (cf.   \cite{Guk,CHNP}). For instance, it is still open to construct compact $G_{2}$ and $Spin(7)$ spaces with isolated singularities (cf.   \cite{Guk}),  which should be an analogue of Theorem \ref{EGZ}.

\section{Collapsing of Calabi-Yau manifolds}
We finish this paper with  a brief  review  of   the collapsing of Ricci-flat Calabi-Yau manifolds,  where our knowledge is more limited comparing to the non-collapsing case.

Let $(X_{k}, \omega_{k})$ be a  sequence of Ricci-flat K\"{a}hler-Einstein Calabi-Yau manifolds with the normalized volume ${\rm Vol}_{\omega_{k}}(X_{k}) \equiv \upsilon$. If
 the diameter of $ \omega_{k}$ tends to infinite when $k \rightarrow \infty$, i.e. $${\rm diam}_{\omega_{k}}(X_{k})  \rightarrow \infty,$$ then  the same as in the proof of  Proposition \ref{dia},  we
obtain $${\rm Vol}_{\omega_{k}} (B_{\omega_{k}}(1)) \leq \frac{8n {\rm
Vol}_{\omega_{k}}(X_{k})}{{\rm diam}_{\omega_{k}}(X_{k})-2}\rightarrow 0, $$ for any  metric 1-ball $B_{\omega_{k}}(1)$,
    by  Theorem 4.1 of Chapter 1 in \cite{SY2} or Lemma 2.3 in \cite{Pa},  i.e.
    $(X_{k}, \omega_{k})$ must collapse.

    Most of the studies  of   the collapsing of Ricci-flat Calabi-Yau manifolds are motivated by the following version of  SYZ conjecture.
Let   $( \mathcal{X} \rightarrow \Delta, \mathcal{L})$ be   a degeneration of polarized Calabi-Yau manifolds of dimension $n$.
  If $0\in \Delta$ is a large complex limit  point (cf. \cite{GHJ}), a refined version of SYZ conjecture  due to Gross, Wilson,  Kontsevich and Soibelman (cf. \cite{GW,KS,KS2})  says that   $$ {\rm diam}_{\omega_{t}}(X_{t}) \sim\sqrt{ -\log |t|}, \  \  {\rm   and} $$  $$(X_{t}, {\rm diam}_{\omega_{t}}^{-2}(X_{t}) \omega_{t})\stackrel{d_{GH}}\longrightarrow (B, d_{B})$$  in the Gromov-Hausdorff sense. If $h^{i,0}(X_{t})=0$, $1\leq i <  n$, then  $B$
 is homeomorphic to $S^{n}$.   Furthermore, there is an open subset $B_{0}\subset B$ with   ${\rm codim}_{\mathbb{R}}B\backslash B_{0} \geqslant 2$, $B_{0}$ admits  a real  affine structure, and the metric $d_{B}$ is induced by a Monge-Amp\`ere  metric $g_{B}$ on $B_{0}$, i.e.  under  affine coordinates $x_{1},  \cdots, x_{n}$,  there is a potential function $\phi$ such that $$g_{B}= \sum_{ij} \frac{ \partial^{2} \phi}{ \partial x_{i}  \partial x_{j}} dx_{i} dx_{j},    \   \  {\rm and}  \    \det (\frac{\partial^{2}\phi}{ \partial x_{i}  \partial x_{j} }) =1.$$ In \cite{KS2},  it is further conjectured that the Gromov-Hausdorff limit  $B$ is homeomorphic to the Calabi-Yau skeleton  of the Berkovich analytic space associated to $\mathcal{X} \times_{\Delta}\Delta^{*}$ by taking some base change if necessary, which gives an algebro-geometric description of $B$.        This conjecture is clearly true for Abelian varieties, and  is verified by Gross and Wilson for fiberd  K3 surfaces with only type $I_{1}$ singular fibers  in \cite{GW}.

 Let $f:X \rightarrow B$ be an   elliptically fibered  K3 surface with 24 type $I_{1}$ singular fibers, and let $\omega_{k}$ be a sequence of Ricci-flat K\"{a}hler-Einstein metrics on $X$ with $\omega_{k}^{2}$ independent of $k$.  If $\epsilon_{k}=\int_{f^{-1}(b)}\omega_{k} \rightarrow 0, $ for a $b\in B$, then \cite{GW} proves that   $$(X_{k}, \epsilon_{k}\omega_{k})
\stackrel{d_{GH}}\longrightarrow (B, d_{B}),  $$ in the Gromov-Hausdorff sense, where  $ d_{B}$ is induced by a special K\"{a}hler metric $g_{B}$ on the complement $B_{0}$  of the discriminant locus.  In a  local coordinate chart  $U\subset B_{0}$, if $f^{-1}(U)=(U\times \mathbb{C})/\Lambda$, where $\Lambda={\rm Span}_{\mathbb{Z}}\{1,\tau\}$ and $\tau$ is a holomorphic function on $U$, i.e. the period of fibers, then ${\rm Im} (\tau)>0$  and  $g_{B}={\rm Im} (\tau) dy \otimes d\bar{y}$.   Furthermore, \cite{GW} shows very explicit asymptotic behaviour of   $\omega_{k}$   when approaching to  the limit.  Near singular fibers, $\epsilon_{k}\omega_{k}$ is asymptotic to the Ooguri-Vafa metric, and  on $f^{-1}(B_{0})$, $\epsilon_{k}\omega_{k}$ is exponentially  asymptotic to the semi-flat metric $\omega_{sf,k}$, i.e.  $$ \|\epsilon_{k}\omega_{k}-\omega_{sf,k}\|_{C^{0}_{loc}}\leq Ce^{-\frac{C}{\epsilon_{k}}},$$ where $$\omega_{sf,k}=\frac{\sqrt{-1}}{2}\Big ({\rm Im} (\tau) dy \wedge d\bar{y}+2\epsilon_{k}^{2}\partial \overline{\partial} \frac{({\rm Im} (z))^{2}}{{\rm Im} (\tau)}\Big  ),$$ and $z$ is the coordinate  on $\mathbb{C}$. By using the HyperK\"{a}hler rotation,  the hypothesis of  this result can be explained as a sequence of complex structures approaching to  the large complex limit while keeping the polarization fixed.  Hence the collapsing version of SYZ conjecture holds for certain  K3 surfaces.

There are generalizations of Gross and  Wilson's theorem to the higher dimensional case. Let $X$ be a Calabi-Yau $n$-manifold admitting a holomorphic fiberation $f:X\rightarrow Z$, where $Z$ is a normal projective variety of dimension $m$, $m<n$. If $Z_{0}\subset Z$ is the complement of the  discriminant locus of $f$, then the fiber  $X_{b}=f^{-1}(b)$, $b\in Z_{0}$, is a Calabi-Yau manifold of dimension $n-m$. Let $\alpha_{Z}$ be an ample class on $Z$, $\alpha_{X}$ be an ample class on $X$, and $\omega_{s} \in \alpha_{Z}+s\alpha_{X}$, $s\in (0,1]$, be the unique Ricci-flat K\"{a}hler-Einstein metric.  The convergence of $\omega_{s}$ is studied in \cite{To2}, and it is shown that $\omega_{s}$ converges to $f^{*}\omega$ in the current sense, where $\omega$ is a K\"{a}hler metric on $Z$ such that the Ricci curvature of $\omega$ is the Weil-Petersson metric for smooth  fibers, i.e. ${\rm Ric}(\omega)=\omega_{WP}$ on $Z_{0}$.  The convergence is improved to be smooth in \cite{GTZ} under the further assumption that general fibers are Abelian varieties. Moreover, we obtain that $(Z_{0}, \omega)$ can be local isometrically  embedded in any Gromov-Hausdorff limit of $(X, \omega_{s})$, and   by using the HyperK\"ahler rotation, we proved a version of SYZ for higher dimensional HyperK\"ahler manifolds in \cite{GTZ}. If the base $Z$ has dimension one,  \cite{GTZ2} shows that $(X, \omega_{s})$ converges to $(Z, \omega)$ in the Gromov-Hausdorff sense, which
 extends Gross-Wilson's result to all elliptically fibred K3 surfaces. Later,   better regularities  of the convergence are  obtained in \cite{HT,TYang} under  more general  assumptions.

 All these progresses are either for the case of a fixed Calabi-Yau manifold, or  that we can use the HyperK\"ahler rotation to convert the question to that case.   The collapsing version of SYZ conjecture  is still widely open beyond the HyperK\"ahler setting. In a recent preprint \cite{zha2}, an analogue  of this conjecture is obtained for canonical polarized manifolds.

\appendix
\section{ Finiteness theorem for   polarized manifolds} $$\text{ by  Valentino Tosatti, Yuguang Zhang}
\footnote{The first-named author is supported in part by NSF grant DMS-1308988 and by a Sloan Research Fellowship.
This work was carried out while the first-named author was visiting the Yau Mathematical Sciences Center of Tsinghua University in Beijing, which he would like to thank for the hospitality. We also thank S. Boucksom for very helpful discussions, and for pointing out that such finiteness results are well-known thanks to \cite{KoMa}.}$$

We show a finiteness theorem for  polarized  complex manifolds.

There are many previous  finiteness theorems about diffeomorphism  types in Riemannian geometry.
    Cheeger's finiteness theorem asserts that given constants  $D$, $\upsilon$, and $\Lambda$, there are only finitely many $n$-dimensional compact differential   manifold $X$ admitting  Riemannian metric $g$ such that ${\rm diam}_{g}(X) \leq D$,  ${\rm Vol}_{ g}( X)\geq \upsilon$ and the sectional curvature $|{\rm Sec}(g)|\leq\Lambda$.  This theorem can be proved as a corollary of the  Cheeger-Gromov convergence  theorem (cf. \cite{Fu,Rong}), which shows that
 if  $(X_{k}, g_{k})$ is a family compact Riemannian
  manifolds with the above bounds, then   a subsequence of  $(X_{k}, g_{k})$ converges to a $C^{1,\alpha}$-Riemannian
  manifold $Y$ in the $C^{1,\alpha}$-sense, and furthermore, $X_{k}$ is diffoemorphic to $Y$ for $k\gg 1$.
In \cite{AC},  Cheeger's finiteness theorem is generalized to the case where the hypothesis on   the sectional curvature bound  is replaced by the weaker bounds of Ricci curvature  $|{\rm Ric}(g)|\leq \lambda$ and the  $L^{\frac{n}{2}}$-norm  of curvature  $\|{\rm Sec}(g)\|_{L^{\frac{n}{2}}}\leq \Lambda$.  Furthermore, if $n=4$ and $g$ is an Einstein metric, then the   integral bound  of curvature can be replaced by a bound for the Euler characteristic.

 We call $(X,L)$  a polarized $n$-manifold, if  $X$ is a compact complex manifold with an ample line bundle $L$. In \cite{KoMa}, a finiteness theorem for  polarized manifolds is obtained. More precisely, Theorem  3 of \cite{KoMa} asserts that for any two constants $V>0$ and $\Lambda >0$,  there are finite many polynomials $P_{1}, \cdots, P_{\ell}$ such that if  $(X,L)$ is  a polarized $n$-manifold with $ c_{1}(L)^{n}\leq V$ and $-c_{1}(X)\cdot c_{1}(L)^{n-1}\leq \Lambda$, then one $P_{i}$  is the Hilbert polynomial of $(X,L)$, i.e. $P_{i}(\nu)=\chi (X,L^{\nu})$.  Consequently, polarized $n$-manifolds with the above  bounds  have only finitely many possible deformation types and finitely many possible   diffeomorphism  types.

 For any constants  $\lambda>0$ and $D>0$,  denote  $$\mathfrak{N}(n, \lambda, D)=\{(X,L)\ |\ \exists \omega\in c_{1}(L) \ {\rm with} \ {\rm Ric}(\omega)\geq -\lambda \omega, \ \ {\rm diam}_{\omega}(X)\leq D \}. $$  Then    $$ c_{1}(L)^{n}=n! {\rm Vol}_{\omega}(X)\leq V=V(n, \lambda,D)$$ by the  Gromov-Bishop comparison theorem, and $$-c_{1}(X)\cdot c_{1}(L)^{n-1}=- \int_{X}{\rm Ric}(\omega)\wedge \omega^{n-1}\leq n \lambda V. $$ The following proposition  is a corollary of Theorem 3 in
 \cite{KoMa}.  Here we give an analytic proof.

\begin{proposition}\label{prop} Polarized manifolds in  $\mathfrak{N}(n, \lambda, D)$  have only finitely many possible Hilbert polynomials, and for any $(X,L)\in \mathfrak{N}(n, \lambda, D)$ we have
\begin{equation}\label{todo}
|\chi (X,L^{\nu})| \leq   C(n,\lambda, D) \nu^{n},
\end{equation}
for all $\nu\geq 1$,
where  $C(n,\lambda, D)$ is a constant depending only on $n, \lambda$ and $D$.
Furthermore,  any $(X,L)\in \mathfrak{N}(n, \lambda, D)$ can be embedded in the same $\mathbb{CP}^{N}$ with   $L^{m}\cong \mathcal{O}_{\mathbb{CP}^{N}}(1)|_{X}$ for integers $m=m(n,  \lambda, D)>0$ and $N=N(n, \lambda, D)>0$.
   As a consequence, manifolds in  $\mathfrak{N}(n, \lambda, D)$ have only finitely many possible deformation types and finitely many possible   diffeomorphism  types.
\end{proposition}

\begin{proof} Let $(X,L)\in \mathfrak{N}(n, \lambda, D)$, and  $\omega\in  c_{1}(L)$ be a K\"{a}hler metric  with $ {\rm Ric}(\omega)\geq -\lambda \omega$, and $ {\rm diam}_{\omega}(X)\leq D$. Fix a Hermitian metric $h$ on $L$ with curvature equal to $\omega$.
The Gromov-Bishop comparison theorem gives   $$1\leq c_{1}(L)^{n}=n! {\rm Vol}_{\omega}(X)\leq V=V(n, \lambda,D).$$
We would  like to  estimate $h^{0,p}(L^\nu)=\dim H^{0,p}(X,L^{\nu})$, $0\leq p \leq n$, for $\nu\geq 1$.
We denote by $\langle \cdot,\cdot\rangle$ the pointwise inner product on $\Omega^{0,p}(X,L^{\nu})$ (smooth $L^\nu$-valued $(0,p)$-forms on $X$) induced by the metric $h^\nu$ on $L^\nu$ whose curvature is  $-\sqrt{-1}\nu \omega$, and by $|\cdot|$ its corresponding norm.
For any $s\in \Omega^{0,p}(X,L^{\nu})$ we have
$$\Delta |s|^2=g^{i\ov{j}}\de_i\de_{\ov{j}}|s|^2=|\nabla s|^2+|\ov{\nabla} s|^2+\langle \Delta s,s\rangle+ \langle s,\ov{\Delta }s\rangle,$$
where $\Delta s=g^{i\ov{j}}\nabla_i\nabla_{\ov{j}}s$ is the rough Laplacian and $\ov{\Delta} s=g^{i\ov{j}}\nabla_{\ov{j}}\nabla_i s$ its ``conjugate''.
Commuting covariant derivatives we get
$$\ov{\Delta }s=\Delta s-\nu n s-{\rm Ric}^{\sharp}(s),$$
where if $p\geq 1$ and we write locally $s=s_{\ov{i_1}\dots\ov{i_p}}d\ov{z}^{i_1}\wedge\dots \wedge d\ov{z}^{i_p}$ with $s_{\ov{i_1}\dots\ov{i_p}}$ local smooth sections of $L^\nu$, then
$${\rm Ric}^{\sharp}(s)=\sum_{j=1}^p g^{k\ov{\ell}}R_{k\ov{i_j}}\ s_{\ov{i_1}\dots \ov{\ell}\dots \ov{i_p}}d\ov{z}^{i_1}\wedge\dots \wedge d\ov{z}^{i_p},$$
while if $p=0$ we let ${\rm Ric}^{\sharp}(s)=0$.
This gives
$$\Delta |s|^2=|\nabla s|^2+|\ov{\nabla} s|^2+2\mathrm{Re}\langle \Delta s,s\rangle-\nu n |s|^2-\langle s, {\rm Ric}^{\sharp}(s)\rangle.$$
Next, we apply the Bochner-Kodaira identity \cite[Theorem 6.2]{MK}, which for any $s\in \Omega^{0,p}(X,L^{\nu})$ gives
$$\Delta_{\db}s=-\Delta s+\nu s +{\rm Ric}^{\sharp}(s),$$
and so if we assume that $\Delta_{\db}s=0$, we obtain
\[\begin{split}
\Delta |s|^2&=|\nabla s|^2+|\ov{\nabla} s|^2+2\langle {\rm Ric}^\sharp(s),s \rangle+2\nu|s|^2-\nu n|s|^2-\langle s, {\rm Ric}^{\sharp}(s)\rangle\\
&=|\nabla s|^2+|\ov{\nabla} s|^2+\langle {\rm Ric}^\sharp(s),s \rangle-\nu (n-2)|s|^2,
\end{split}\]
noting that $\langle {\rm Ric}^\sharp(s),s \rangle=\langle s, {\rm Ric}^{\sharp}(s)\rangle$.
Using that
$$\langle {\rm Ric}^\sharp(s),s \rangle\geq -\lambda p |s|^2,$$
we finally obtain
$$\Delta |s|^2\geq -(\nu (n-2)+\lambda p)|s|^2.$$
A standard Moser iteration argument (see e.g. \cite[Lemma 2.4]{Bour}) applied to this differential inequality gives
\begin{equation}\label{estim}
\sup_X|s|^2\leq A (\nu (n-2)+\lambda p)^{n}\int_X|s|^2\frac{\omega^n}{n!}=A(\nu (n-2)+\lambda p)^n\|s\|^2_{L^{2}},
\end{equation}
where $A$ depends only on the Sobolev constant of $\omega$ and on $n$. Thus $A=A(n, V,\lambda, D)$  by a result of Croke \cite{Cr}.

  Now we use the arguments in Lemma 11 and Theorem 12 of the paper of Li \cite{Li}.
  By the Hodge Theorem, we have an isomorphism $H^{0,p}(X,L^{\nu})\cong \mathcal{H}^{0,p}(X,L^{\nu})$, the space of $\Delta_{\db}$-harmonic forms in $\Omega^{0,p}(X,L^{\nu})$.
Let $$\rho=\sum |s_{i}|^{2}$$ for  an
 orthonormal
 basis $s_{i}$
of $\mathcal{H}^{0,p}(X,L^{\nu})$. The function $\rho$ is easily seen to be independent of the choice of orthonormal basis. Let $x\in X$ such that $$\rho(x)=\sup_{X} \rho >0.$$ Then
$$E_{0}=\{s\in \mathcal{H}^{0,p}(X,L^{\nu})| s(x)=0\},$$
is a proper linear subspace of $\mathcal{H}^{0,p}(X,L^{\nu})$, with orthogonal complement $E_{0}^{\bot}$.
 We claim that $\dim E_{0}^{\bot}\leq \binom{n}{p}$. If $s_{1}, \cdots, s_{r}$, $r>\binom{n}{p}$,   is an
 orthonormal
 basis of $E_{0}^{\bot}$, then there are $a_{i}$, $i=1, \cdots, r$,  such that $\sum a_{i}s_{i}(x)=0$.  Thus $\sum a_{i}s_{i}\in E_{0}$, which  is a contradiction.

   Let $s_{1}, \cdots, s_{r}\in \mathcal{H}^{0,p}(X,L^{\nu})$ be an orthonormal basis of $ E_{0}^{\bot}$, which we can complete to an orthonormal basis of $\mathcal{H}^{0,p}(X,L^{\nu})$
with an orthonormal basis $s_{r+1}, \cdots, s_{N}$ of $E_{0}$.  We have
\[\begin{split}
h^{0,p}(L^\nu)= \int_{X}\rho \frac{\omega^n}{n!}\leq  V\sup_X\rho& = V \sup_X\left(\sum_{i=1}^{r}|s_{i}|^{2}\right)\\
&\leq \binom{n}{p} V \sup_i\|s_i\|_{L^{\infty}}^{2}\\  &\leq \binom{n}{p} V  A (\nu (n-2)+\lambda p)^{n},
\end{split}\]
using \eqref{estim},
and thus
 for any $\nu \geq 1$ we have
$$|\chi (X,L^{\nu})|=\left|\sum_{p}(-1)^{p}h^{0,p}(L^\nu)\right|\leq \sum_{p}  \binom{n}{p} V  A (\nu (n-2)+\lambda p)^{n}\leq C(n,\lambda,D)\nu^n,$$
thus proving \eqref{todo}.
Since
the Hilbert polynomial $P$ of $(X,L)$ is given by
$$P(\nu)=\chi (X,L^{\nu})=\int_{X}e^{\nu c_{1}(L)}{\rm Todd}_{X}=a_{0}\nu^{n}+a_{1}\nu^{n-1}+ \cdots +a_{n}\in \mathbb{Z},$$
 it follows that we have only finitely many possible values for $a_{0}, \cdots, a_{n}$ by taking sufficiently many values of $\nu$ and solving the linear equations. Hence $(X,L)$ has only finitely many possible Hilbert polynomials.

 Now, by Matsusaka's Big Theorem (cf. \cite{Mat}),       there is an  $m_{0}>0$ depending only on $P$    such that  for any $m\geqslant m_{0}$,  $L^{m}$ is very ample, and $H^{i}(X,L^{m})=\{0\}$, $i>0$.   By choosing a basis  $\Sigma$ of $H^{0}(X,L^{m})$,  we have an embedding $\Phi_{\Sigma}: X
 \hookrightarrow \mathbb{CP}^{N}$ such that  $L^{m}=\Phi_{\Sigma}^{*}\mathcal{O}_{\mathbb{CP}^{N}}(1)$. We   regard $\Phi_{\Sigma}(X)$ as a   point in the Hilbert scheme  $\mathcal{H}ilb_{N}^{P_{m}}$ parametrizing the subshemes of $\mathbb{CP}^{N}$ with   Hilbert polynomial $P_{m}(\nu)=P(m\nu)$, where $N=h^{0} (X,L^{m})-1$.
 Finally, $\Phi_{\Sigma}( X)$ belongs to finitely many possible components of finitely many possible Hilbert schemes, and thus $\mathfrak{N}(n, \lambda, D)$ has only finitely many possible  deformation and  diffeomorphism  types.
  \end{proof}

Note that for any polarized manifold  $(X,L)$, the  volume of $(X,\omega)$ as in the definition of $\mathfrak{N}(n, \lambda, D)$ is bounded below uniformly away from zero. We remark that a similar diffeomorphism finiteness result fails for the family of closed Riemannian manifolds $(M,g)$ of real dimension $m$ with ${\rm Ric}(g)\geq -\lambda g, {\rm diam}_g(X)\leq D, {\rm vol}_g(X)\geq v>0$. Indeed Perelman \cite{Per} constructed Riemannian metrics on $\sharp_{k}\mathbb{CP}^2$ for all $k\geq 1$, which have positive Ricci curvature, unit diameter and volume bounded uniformly away from zero.

\end{document}